\pdfoutput=1 
\documentclass[a4paper]{article}
\usepackage{amsthm,amssymb,amsmath,enumerate,graphicx}
\usepackage[utf8]{inputenc}

\newtheorem{THM}{Theorem}[section]

\newtheorem{LEM}[THM]{Lemma}

\newtheorem{COR}[THM]{Corollary}
\newtheorem{PROP}[THM]{Proposition}

\theoremstyle{definition}
\newtheorem{EX}[THM]{Example}

\def\shift(#1)(#2){\!\!\downarrow\!{}^{#1}_{\raise .1ex\vbox to 0pt{\vss\hbox{$\scriptstyle #2$}}}\,}
\def\ucl(#1){\lfloor #1 \rfloor}
\def\dcl(#1){\lceil #1 \rceil}
\def\specrel#1#2{\mathrel{\mathop{\kern0pt #1}\limits_{#2}}}
\def\restricts{\!\restriction\!}

\newcommand\T{\mathcal T}

\def\F{\mathcal F}
\def\L{\mathcal L}
\def\S{\mathcal S}
\def\X{\mathcal X}

\def\lowfwd #1#2#3{{\mathop{\kern0pt #1}\limits^{\kern#2pt\raise.#3ex
\vbox to 0pt{\hbox{$\scriptscriptstyle\rightarrow$}\vss}}}}
\def\lowbkwd #1#2#3{{\mathop{\kern0pt #1}\limits^{\kern#2pt\raise.#3ex
\vbox to 0pt{\hbox{$\scriptscriptstyle\leftarrow$}\vss}}}}
\def\lowfbwd #1#2#3{{\mathop{\kern0pt #1}\limits^{\kern#2pt\raise.#3ex
\vbox to 0pt{\hbox{$\scriptscriptstyle\leftrightarrow$}\vss}}}}
\def\fwd #1#2{{\lowfwd{#1}{#2}{15}}}
\def\ve{\kern-1.5pt\lowfwd e{1.5}2\kern-1pt}
\def\vedash{{\mathop{\kern0pt e\lower.5pt\hbox{${}
     \scriptstyle'$}}\limits^{\kern0pt\raise.02ex
     \vbox to 0pt{\hbox{$\scriptscriptstyle\rightarrow$}\vss}}}}
\def\ev{\kern-1pt\lowbkwd e{0.5}2\kern-1pt}
\def\vf{\kern-2pt\lowfwd f{2.5}2\kern-1pt}
\def\vfdash{{\mathop{\kern0pt f\raise 1pt\hbox{${}
     \scriptstyle'$}}\limits^{\kern2pt\raise.02ex
     \vbox to 0pt{\hbox{$\scriptscriptstyle\rightarrow$}\vss}}}}

\def\vp{\lowfwd p{1.5}2}

\def\vr{\lowfwd r{1.5}2}
\def\rv{\lowbkwd r02}

\def\vrdash{{\mathop{\kern0pt r\lower.5pt\hbox{${}
     \scriptstyle'$}}\limits^{\kern0pt\raise.02ex
     \vbox to 0pt{\hbox{$\scriptscriptstyle\rightarrow$}\vss}}}}
\def\rvdash{{\mathop{\kern0pt r\lower.5pt\hbox{${}
     \scriptstyle'$}}\limits^{\kern0pt\raise.02ex
     \vbox to 0pt{\hbox{$\scriptscriptstyle\leftarrow$}\vss}}}}
\def\vrdashp{{\mathop{\kern0pt r_p\kern-4pt\lower.5pt\hbox{${}
     \scriptstyle'$}}\limits^{\kern0pt\raise.02ex
     \vbox to 0pt{\hbox{$\scriptscriptstyle\rightarrow$}\vss}}}\,}
\def\rvdashp{{\mathop{\kern0pt r_p\kern-4pt\lower.5pt\hbox{${}
     \scriptstyle'$}}\limits^{\kern0pt\raise.02ex
     \vbox to 0pt{\hbox{$\scriptscriptstyle\leftarrow$}\vss}}}\,}
\def\vrddash{{\mathop{\kern0pt r\lower.5pt\hbox{${}
     \scriptstyle''$}}\limits^{\kern0pt\raise.02ex
     \vbox to 0pt{\hbox{$\scriptscriptstyle\rightarrow$}\vss}}}}
\def\vrone{\lowfwd {r_1}12}

\def\vrtwo{\lowfwd {r_2}12}

\def\vrp{\lowfwd {r_p}12}
\def\rvp{\lowbkwd {r_p}02}
\def\vrq{\lowfwd {r_q}12}
\def\rvq{\lowbkwd {r_q}02}

\def\amgis{\lowbkwd \sigma02}
\def\vs{\lowfwd s{1.5}1}
\def\sv{{{\lowbkwd s{1.5}1}\hskip-1pt}}
\def\vsdash{{\mathop{\kern0pt s\lower.5pt\hbox{${}
     \scriptstyle'$}}\limits^{\kern0pt\raise.02ex
     \vbox to 0pt{\hbox{$\scriptscriptstyle\rightarrow$}\vss}}}}
\def\svdash{{\mathop{\kern0pt s\lower.5pt\hbox{${}
     \scriptstyle'$}}\limits^{\kern0pt\raise.02ex
     \vbox to 0pt{\hbox{$\scriptscriptstyle\leftarrow$}\vss}}}}
\def\vsddash{{\mathop{\kern0pt s\lower.5pt\hbox{${}
     \scriptstyle''$}}\limits^{\kern0pt\raise.02ex
     \vbox to 0pt{\hbox{$\scriptscriptstyle\rightarrow$}\vss}}}}
\def\svddash{{\mathop{\kern0pt s\lower.5pt\hbox{${}
     \scriptstyle''$}}\limits^{\kern0pt\raise.02ex
     \vbox to 0pt{\hbox{$\scriptscriptstyle\leftarrow$}\vss}}}}
\def\vsdashp{{\mathop{\kern0pt s_p\kern-4pt\lower.5pt\hbox{${}
     \scriptstyle'$}}\limits^{\kern0pt\raise.02ex
     \vbox to 0pt{\hbox{$\scriptscriptstyle\rightarrow$}\vss}}}\,}
\def\svdashp{{\mathop{\kern0pt s_p\kern-4pt\lower.5pt\hbox{${}
     \scriptstyle'$}}\limits^{\kern0pt\raise.02ex
     \vbox to 0pt{\hbox{$\scriptscriptstyle\leftarrow$}\vss}}}\,}
\def\vso{\lowfwd {s_0}11}
\def\svo{\lowbkwd {s_0}02}
\def\vsone{\lowfwd {s_1}11}
\def\svone{\lowbkwd {s_1}11}
\def\vstwo{\lowfwd {s_2}11}

\def\vsi{\lowfwd {s_i}11}

\def\vsn{\lowfwd {s_n}11}

\def\vsidash{{\mathop{\kern0pt s_i\kern-3.5pt\lower.3pt\hbox{${}
     \scriptstyle'$}}\limits^{\kern0pt\raise.02ex
     \vbox to 0pt{\hbox{$\scriptscriptstyle\rightarrow$}\vss}}}}

\def\vsp{\lowfwd {s_p}11}
\def\vspdash{\lowfwd {s_{p'}}11}
\def\svp{\lowbkwd {s_p}{-1}2}
\def\vsq{\lowfwd {s_q}11}
\def\svq{\lowbkwd {s_q}{-1}2}
\def\vsqdash{{\mathop{\kern0pt s_q\kern-3.5pt\lower.3pt\hbox{${}
     \scriptstyle'$}}\limits^{\kern0pt\raise.02ex
     \vbox to 0pt{\hbox{$\scriptscriptstyle\rightarrow$}\vss}}}}
\def\vtp{\lowfwd {t_p}11}
\def\vtq{\lowfwd {t_q}11}
\def\tvp{\lowbkwd {t_p}01}

\def\vR{{\hskip-1pt{\fwd R3}\hskip-1pt}} 
\def\vS{{\hskip-1pt{\fwd S3}\hskip-1pt}} 
\def\vSp{\lowfwd {S_p}11}

\def\vSpnought{\lowfwd {S_{p_0}}11}
\def\vSq{\lowfwd {S_q}11}
\def\vSstar{{\mathop{\kern0pt S\lower-1pt\hbox{$^*$}}\limits^{\kern2pt
     \vbox to 0pt{\hbox{$\scriptscriptstyle\rightarrow$}\vss}}}}
\def\vSdash{{\mathop{\kern0pt S\lower-1pt\hbox{${}
     \scriptstyle'$}}\limits^{\kern2pt\raise.1ex
     \vbox to 0pt{\hbox{$\scriptscriptstyle\rightarrow$}\vss}}}}

\def\vSinf{\vS}
\def\vt{\lowfwd t{1.5}1}
\def\tv{\lowbkwd t{1.5}1}

\def\es{\emptyset}
\def\sub{\subseteq}
\def\supe{\supseteq}
\def\sm{\smallsetminus}

\newcommand\COMMENT[1]{}

\def\?#1{\vadjust{\vbox to 0pt{\vss\vskip-8pt\leftline{%
     \llap{\hbox{\vbox{\pretolerance=-1
     \doublehyphendemerits=0\finalhyphendemerits=0
     \hsize20truemm\tolerance=10000\small
     \lineskip=0pt\lineskiplimit=0pt
     \rightskip=0pt plus16truemm\baselineskip8pt\noindent
     \hskip0pt        
     #1\endgraf}\hskip2truemm}}}\vss}}}

\newenvironment{txteq*}
  {
    \begin{equation*}
    \begin{minipage}[c]{0.85\textwidth} 
    \em                                
  }
  {\end{minipage}\end{equation*}\ignorespacesafterend}

\title{Profinite separation systems\\ {\small Extended version}}
 \author{Reinhard Diestel\and Jakob Kneip} %
   
\renewcommand\S{\mathcal S}

\def\F{\mathcal F}

\newcommand{\braces}[1]{\left(#1\right)}

\newcommand{\abs}[1]{\left |#1\right |}

\newcommand{\tn}[1]{\textnormal{#1}}
\def\N{\mathbb{N}}
\def\proj{\restricts}

\def\lowfwd #1#2#3{{\mathop{\kern0pt #1}\limits^{\kern#2pt\raise.#3ex
\vbox to 0pt{\hbox{$\scriptscriptstyle\rightarrow$}\vss}}}}
\def\lowbkwd #1#2#3{{\mathop{\kern0pt #1}\limits^{\kern#2pt\raise.#3ex
\vbox to 0pt{\hbox{$\scriptscriptstyle\leftarrow$}\vss}}}}
\def\fwd #1#2{{\lowfwd{#1}{#2}{15}}}
\def\ve{\kern-1pt\lowfwd e{1.5}2\kern-1pt}
\def\ev{\kern-1pt\lowbkwd e{1.5}2\kern-1pt}
\def\vp{\lowfwd p{1.5}2}

\def\vr{\lowfwd r{1.5}2}
\def\rv{\lowbkwd r02}

\def\vrdash{{\mathop{\kern0pt r\lower.5pt\hbox{${}
     \scriptstyle'$}}\limits^{\kern0pt\raise.02ex
     \vbox to 0pt{\hbox{$\scriptscriptstyle\rightarrow$}\vss}}}}
\def\rvdash{{\mathop{\kern0pt r\lower.5pt\hbox{${}
     \scriptstyle'$}}\limits^{\kern0pt\raise.02ex
     \vbox to 0pt{\hbox{$\scriptscriptstyle\leftarrow$}\vss}}}}
\def\vrone{\lowfwd {r_1}12}

\def\vrtwo{\lowfwd {r_2}12}

\def\vs{\lowfwd s{1.5}1}
\def\sv{\lowbkwd s{1.5}1}
\def\vso{\lowfwd {s_0}11}
\def\svo{\lowbkwd {s_0}02}
\def\vsone{\lowfwd {s_1}11}
\def\vstwo{\lowfwd {s_2}11}

\def\vsi{\lowfwd {s_i}11}
\def\vsidash{{\mathop{\kern0pt s_i\kern-3.5pt\lower.3pt\hbox{${}
     \scriptstyle'$}}\limits^{\kern0pt\raise.02ex
     \vbox to 0pt{\hbox{$\scriptscriptstyle\rightarrow$}\vss}}}}
\def\vS{{\hskip-1pt{\fwd S3}\hskip-1pt}} 
\def\vSr{{\vec S}_{\raise.1ex\vbox to 0pt{\vss\hbox{$\scriptstyle\ge\vr$}}}}
\def\vSdash{{\mathop{\kern0pt S\lower-1pt\hbox{${}
     \scriptstyle'$}}\limits^{\kern2pt\raise.1ex
     \vbox to 0pt{\hbox{$\scriptscriptstyle\rightarrow$}\vss}}}}
\def\vsdash{{\mathop{\kern0pt s\lower.5pt\hbox{${}
     \scriptstyle'$}}\limits^{\kern0pt\raise.02ex
     \vbox to 0pt{\hbox{$\scriptscriptstyle\rightarrow$}\vss}}}}
     
\def\svdash{{\mathop{\kern0pt s\lower.5pt\hbox{${}
     \scriptstyle'$}}\limits^{\kern0pt\raise.02ex
     \vbox to 0pt{\hbox{$\scriptscriptstyle\leftarrow$}\vss}}}}
     
\def\vtdash{{\mathop{\kern0pt t\lower0pt\hbox{${}
     \scriptstyle'$}}\limits^{\kern0pt\raise.1ex
     \vbox to 0pt{\hbox{$\scriptscriptstyle\rightarrow$}\vss}}}}
     
\def\tvdash{{\mathop{\kern0pt t\lower0pt\hbox{${}
     \scriptstyle'$}}\limits^{\kern0pt\raise.1ex
     \vbox to 0pt{\hbox{$\scriptscriptstyle\leftarrow$}\vss}}}}
     
\def\vddash{{\mathop{\kern0pt d\raise1pt\hbox{${}
     \scriptstyle'$}}\limits^{\kern0pt\raise.02ex
     \vbox to 0pt{\hbox{$\scriptscriptstyle\rightarrow$}\vss}}}}
     
\def\dvdash{{\mathop{\kern0pt d\raise1pt\hbox{${}
     \scriptstyle'$}}\limits^{\kern0pt\raise.02ex
     \vbox to 0pt{\hbox{$\scriptscriptstyle\leftarrow$}\vss}}}}
     
\def\vtstar{{\mathop{\kern0pt t\raise2.5pt\hbox{${}
     \scriptstyle*$}}\limits^{\kern0pt\raise.1ex
     \vbox to 0pt{\hbox{$\scriptscriptstyle\rightarrow$}\vss}}}}
     
\def\tvstar{{\mathop{\kern0pt t\raise2.5pt\hbox{${}
     \scriptstyle*$}}\limits^{\kern0pt\raise.1ex
     \vbox to 0pt{\hbox{$\scriptscriptstyle\leftarrow$}\vss}}}}

\def\vtstarD{{\mathop{\kern0pt t\kern.5pt\raise3pt\hbox{${}
     \scriptstyle*$}{\kern-5.5pt\lower3pt\hbox{$
     \scriptstyle D$}}}\limits^{\kern0pt\raise.1ex
     \vbox to 0pt{\hbox{$\scriptscriptstyle\rightarrow$}\vss}}}}

\def\tvstarD{{\mathop{\kern0pt t\kern.5pt\raise3pt\hbox{${}
     \scriptstyle*$}{\kern-5.5pt\lower3pt\hbox{$
     \scriptstyle D$}}}\limits^{\kern0pt\raise.1ex
     \vbox to 0pt{\hbox{$\scriptscriptstyle\leftarrow$}\vss}}}}

\def\vt{\lowfwd t{1.5}1}
\def\tv{\lowbkwd t{1.5}1}

\def\vm{\lowfwd m{1}2}
\def\mv{\lowbkwd m02}
\def\vn{\lowfwd n{1.5}2}
\def\nv{\lowbkwd n{.5}2}
\def\vt{\lowfwd t{1.5}1}
\def\tv{\lowbkwd t{1.5}1}

\def\vx{\lowfwd x{1.5}1}

\def\vE{\lowfwd E{1.5}1}

\def\invlim{\varprojlim\,\,}

 \def\vSdashp{{\mathop{\kern0pt S_p\kern-4pt\lower-1pt\hbox{$\scriptstyle'$}}\limits^{\kern1pt\raise.1ex
			\vbox to 0pt{\hbox{$\scriptscriptstyle\rightarrow$}\vss}}}}

\begin{document}
\abovedisplayshortskip=-3pt plus3pt
\belowdisplayshortskip=6pt

\maketitle

\begin{abstract}\noindent
Separation systems are posets with additional structure that form an abstract setting in which tangle-like clusters in graphs, matroids and other combinatorial structures can be expressed and studied.

This paper offers some basic theory about infinite separation systems and how they relate to the finite separation systems they induce. They can be used to prove tangle-type duality theorems for infinite graphs and matroids, which will be done in future work that will build on this paper.
\end{abstract}

\section{Introduction}\label{sec:intro}

This paper is a sequel to, and assumes familiarity with, an earlier paper~\cite{AbstractSepSys} in which finite abstract separation systems were introduced. The profinite separation systems introduced here, and the results proved, will form the basis for our proof of the tangle duality theorem for infinite graphs~\cite{duality1inf}, as well as for a more comprehensive study of profinite {\em tree sets\/}~\cite{KneipProfiniteTreeSets}.

Abstract separation systems were introduced in~\cite{AbstractSepSys} to lay the foundations for a comprehensive study of how tangles, originally introduced by Robertson and Seymour~\cite{GMX} in the course of their graph minors project, can be generalized to capture, and relate, various other types of highly cohesive regions in graphs, matroids and other combinatorial structures. The basic idea behind tangles is that they describe such a region indirectly: not by specifying which objects, such as vertices or edges, belong to it, but by setting up a system of pointers on the entire structure that point towards this region. The advantage of this indirect approach is that such pointers can locate such a highly cohesive region even when it is a little fuzzy~-- e.g., when for every low-order separation of a graph or matroid `most' of the region will lie on one side or the other, so that this separation can be oriented towards it and become a pointer, but each individual vertex or edge (say) can lie on the `wrong' side of {\em some\/} such separation. (The standard example is that of a large grid in a graph: for every low-oder separation of the graph, most of the grid will lie on one side, but each of its vertices lies on the `other' side of the small separator that consists of just its four neighbours.)

This indirect approach can be applied also to capture fuzzy clusters in settings very different from graphs; see \cite{TanglesEmpirical, TanglesSocial, ProfilesNew, TangleTreeAbstract, MonaLisa} for more discussion.

Robertson and Seymour~\cite{GMX} proved two major theorems about tangles: the {\em tangle-tree theorem\/}, which shows how the tangles in a graph can be pairwise separated by a small set of nested separations (which organize the graph into a tree-like shape), and the {\em tangle duality\/} theorem which tells us that if the graph has no tangle of a given `order' then the entire graph has corresponding tree-like structure. Both these theorems can be proved in the general settings of abstract separation systems~\cite{TangleTreeAbstract, ProfilesNew}, and then be applied in the above scenarios.

Unlike the tangle-tree theorem, the tangle duality theorem can be extended without unforeseen obstacles to infinite graphs and matroids. This too is most conveniently done in the setting of abstract separation systems. The infinite separation systems needed for this are {\em profinite\/}~-- which means, roughly, that they are completely determined by their finite subsystems. This helps greatly with the proof: it allows us to use the abstract tangle duality theorems from~\cite{TangleTreeAbstract} on the finite subsystems, and then extend this by compactness~\cite{duality1inf}.

Although compactness lies at the heart of these extensions, they are not entirely straightforward: one needs to know how exactly a profinite separation system is related to its finite subsystems~-- not just that these, in principle, determine it. This paper examines and solves these questions. It does so in the general setting of arbitrary profinite separation systems, of which the separations of infinite graphs and matroids (no matter how large) are examples.%
   \footnote{Note that whether or not two vertex sets $A,B$ covering a graph~$G$ form a separation is determined by its subgraphs of order~2: it is if and only if $G$ has no edge from $A\sm B$ to $B\sm A$.\looseness=-1}

Our paper is organized as follows. The basic definitions and facts about abstract separation systems are not repeated here; instead, the reader is referred to~\cite{AbstractSepSys}. In Section~\ref{sec:separations} we cite just a couple of lemmas that will be used throughout, and define homomorphisms between separation systems; these will be needed to set up the inverse systems on which our profinite separation systems will be based. A~brief reminder of inverse limits of finite sets is offered in Section~\ref{sec:limits}. Profinite separation systems, including a natural topology on them, are then introduced rigorously in Section~\ref{sec:limsep}.\ Section~\ref{sec:limtreesets} contains most of our groundwork: it gives a detailed analysis of how, in `nested' profinite separation systems, the properties of the separation system itself are related to the properties of the finite separation systems of which it is an inverse limit. In Section~\ref{sec:compactnessthm} we apply this to obtain a compactness theorem for subsystems that occur as witnesses to the nonexistence of abstract tangles~\cite{duality1inf}. The theorem says that a profinite separation system has a nested subset of a certain type~-- one that avoids some given forbidden sets of separations~-- as soon as all its finite projections have nested subsets of that type. In Section~\ref{sec:orientations}, finally, we compare the combinatorial properties of nested separation systems with their topological properties; this has implications on the tools, topological or combinatorial, that can be used to handle profinite nested separation systems when these are applied.

\section{Separation Systems}\label{sec:separations}

For definitions and basic properties of abstract separation systems we refer to \cite{AbstractSepSys} and \cite{TreeSets}. In addition, we shall use the following terms.

We call a separation {\em co-small\/} if its inverse is small. If an oriented separation $\vr$ is trivial and this is witnessed by some separation~$s$, we also call the orientations $\vs,\sv$ of~$s$ {\em witnesses\/} of the triviality of~$\vr$. And finally, if $\vr\in\sigma\sub\vS$ for some separation system~$\vS$, we may informally say that $\vr$ is {\em trivial in~$\sigma$} to express that $\vr$ is trivial in~$\vS$ with a witness in~$\sigma$. Note that this does not depend on anything in~$\vS$ outside~$\sigma$: it is equivalent to saying that $\vr$ is trivial in the separation system that $\vS$ induces on the set of separations in $\sigma$ and their inverses.

We shall be using the following two lemmas from~\cite{AbstractSepSys}:

\begin{LEM}[Extension Lemma]\tn{\cite[Lemma 4.1]{AbstractSepSys}}\label{lem:extension}
	Let $S$ be a set of unoriented separations, and let $P$ be a consistent partial orientation of $S$.
	\begin{enumerate}\itemsep=0pt
		\item[\tn{(i)}]$P$ extends to a consistent orientation~$ O $ of $S$ if and only if no element of $P$ is co-trivial in $S$.
		\item[\tn{(ii)}]If $\vp$ is maximal in $P$, then $O$ in \tn{(i)} can be chosen with $\vp$ maximal in $O$ if and only if $\vp$ is nontrivial in $\vS$.
		\item[\tn{(iii)}]If $S$ is nested, then the orientation~$ O $ in \tn{(ii)} is unique.
	\end{enumerate}
\end{LEM}

\goodbreak

Nested separation systems without degenerate or trivial elements are known as {\em tree sets\/}. A subset $\sigma\sub\vS$ {\em splits\/} a nested separation system~$\vS$ if $ \vS $ has a consistent orientation $ O $ such that $ O\subseteq\dcl(\sigma) $ and $ \sigma $ is precisely the set of maximal elements of~$ O $. Conversely, a consistent orientation $O$ of~$\vS$ {\em splits} (at~$\sigma$) if it is contained in the down-closure of the set~$\sigma$ of its maximal elements.

The consistent orientations of a finite nested separation system can be recovered from its splitting subsets by taking their down-closures, but infinite separation systems can have consistent orientations without any maximal elements.%
   \COMMENT{}

\begin{LEM}\tn{\cite[Lemmas 4.4, 4.5]{AbstractSepSys}}\label{Remark8}
	The splitting subsets of a nested separation system $\vS$ without degenerate elements are proper stars. Their elements are neither trivial nor co-trivial in~$\vS$. If $S$ has a degenerate element~$s$, then $\{\vs\}$ is the unique splitting subset of~$\vS$.
\end{LEM}

The splitting stars of the edge tree set $\vE(T) $ of a tree~$ T$, for example, are the sets $\vec F_t$ of edges at a node~$t$ all oriented towards~$t$. By this correspondence, the nodes of~$T$ can be recovered from its edge tree set; if $T$ is finite, its nodes correspond bijectivly to its consistent orientations.

Given two separation systems $ R,S $, a map $ f\colon \vR\to \vS $ is a {\it homomorphism} of separation systems if it commutes with their involutions and respects the ordering on~$\vR$. Formally, we say that $ f $ \textit{commutes with the involutions} of $\vR$ and~$\vS$ if $ \braces{f(\vr)}^*=f(\rv) $ for all ${\vr\in \vR}$. It \textit{respects the ordering} on~$\vR$ if $ f({\vrone})\le f({\vrtwo}) $ whenever $ {\vrone}\le{\vrtwo} $. Note that the condition for $ f $ to be order-respecting is not `if and only if': we allow that $ f({\vrone})\le f({\vrtwo}) $ also for incomparable~$ {\vrone},{\vrtwo}\in R $. Furthermore~$ f $ need not be injective. It can therefore happen that ${\vrone < \vrtwo}$ but $f(\vrone) = f(\vrtwo)$, so $f$ need not preserve strict inequality. A~bijective homomorphism of separation systems whose inverse is also a homomorphism is an {\em isomorphism\/}.

\section{Inverse limits of sets}\label{sec:limits}

Let $(P,\le)$ be a {\em directed} poset, one such that for all $p,q\in P$ there is an $r\in P$ such that $p,q\le r$. A~family $\X = (\,X_p\mid p\in P\,)$ of sets $X_p$ is an {\em inverse system\/} if it comes with maps $f_{qp}\colon X_q\to X_p$, defined for all $q > p$, which are compatible in that $f_{rp} = f_{qp}\circ f_{rq}$ whenever $p < q < r$.%
   \footnote{If desired, this can be extended to all $p\le q\le r$ in~$P$ by letting $f_{pp} = {\rm id}_{X_p}$ for all $p\in P$.}
   If all the~$f_{qp}$ are surjective, we call~$\X$ a {\em surjective inverse system\/}. Any inverse system $(\,Y_p\mid p\in P\,)$ with $Y_p\sub X_p$ and maps $f_{qp}\restricts Y_q$ is a {\em restriction\/} of~$\X$.

A family $(\,x_p\mid p\in P\,)$ with $x_p\in X_p$ such that $f_{qp}(x_q) = x_p$ for all $p < q$ is a {\em limit} of~$\X$. The set $\varprojlim\X$ of all limits of $\X$ is the {\em inverse limit} of~$\X$. If each of the $X_p$ carries a topology, then $\varprojlim\X$ is a subspace of the product space $\prod_{p\in P} X_p$, and we give $\X$ the subspace topology of the product topology on this space. In this paper, all the sets $X_p$ will be finite and carry the discrete topology. This makes the maps $f_{qp}$ continuous and $\prod_{p\in P} X_p$ compact, so $\X$ is compact too since it is closed in~$\prod_{p\in P} X_p$.%
   \COMMENT{}

The following folklore `compactness theorem' ensures that limits of such inverse systems exist; see, e.g., \cite{ProfiniteGroups} for an introduction to inverse limits including its standard short proof.%
   \COMMENT{}

\begin{LEM}\label{CompactnessThm}
The inverse limit of any inverse system of non-empty discrete finite sets is non-empty, Hausdorff and compact.
\end{LEM}

Thus, in any inverse system of finite sets we can pick one element from each of them so that these choices commute with the maps~$f_{qp}$. When $P$ is the set of natural numbers, this becomes K\"onig's familiar infinity lemma.

\section{Profinite separation systems}\label{sec:limsep}

For this entire section, let $P$ be a directed poset and $\S = (\,\vSp\mid p\in P\,)$ a family of finite separation systems%
   \COMMENT{}
   $(\vSp,\le,{}^*)$ with homomorphisms $f_{qp}\colon \vSq\to\vSp$ for all $p<q$. Then $\vSinf := \varprojlim\S$ becomes a separation system $(\vSinf,\le,{}^*)$ if we define
  \begin{itemize}
\item $(\,\vrp\mid p\in P\,) \le (\,\vsp\mid p\in P\,)\> :\Leftrightarrow\ \forall p\in P\colon \vrp\le\vsp$\quad and
\item $(\,\vsp\mid p\in P\,)^* := (\,\svp\mid p\in P\,)$.%
   \COMMENT{}
\end{itemize}
We call $(\vSinf,\le,{}^*)$ the {\em inverse limit\/} of the separation systems $(\vSp,\le,{}^*)$. Separation systems that are isomorphic to the inverse limit of some finite separation systems are {\em profinite}.

For example, let $V$ be any set, $(\vS_V,\le,{}^*)$ the system of set separations of~$V\!$, and suppose that $P$ is the set of finite subsets of~$V\!$ ordered by inclusion.%
   \COMMENT{}
    For each $p\in P$ assume that $(\vSp,\le,{}^*)$ is the system of all the set separations of~$p$, and that the $f_{qp}$ are restrictions: that for $q>p$ and $\vsq = (A,B)\in\vSq$ we have
 $$f_{qp}(\vsq) = \vsq\restricts p := (A\cap p, B\cap p).$$
 Then every limit $\vs = (\,\vsp\mid p\in P\,) \in \vS = \invlim\S$ defines an oriented separation $(A,B)$ of~$V$ by setting $A:= \bigcup_{p\in P} A_p$ and $B:= \bigcup_{p\in P} B_p$ for $\vsp  = :(A_p,B_p)$. This $(A,B)$ is the unique separation of~$V$ such that $(A\cap p, B\cap p) = \vsp$ for all~$p$, and $\vs\mapsto (A,B)$ is easily seen to be an isomorphism between the separation systems $\vSinf$ and~$\vS_V$.

Although we shall mostly work with abstract separation systems and their inverse limits in this paper, it will be useful to keep this example in mind. To help support this intuition, we shall usually write the maps $f_{qp}$ as restrictions, i.e., write
 $$\vsq\restricts p := f_{qp}(\vsq)$$
 whenever $\vsq\in\vSq$ and $q>p$, as well as
 $$\vs\restricts p := \vsp$$
 whenever $\vs = (\,\vsp\mid p\in P\,)\in\vSinf$, even when the separations considered are not separations of sets. Similarly, given $s\in S$, or $s\in S_q$ with $q\ge p$, we write $s\restricts p := \{\vsp, \svp\}$ for the separation in~$S_p$ with orientations $\vsp := \vs\restricts p$ and $\svp := \sv\restricts p$.

Given a subset $O\sub\vSinf$, we write
 $$O\restricts p := \{\,\vs\restricts p\mid \vs\in O\,\}.$$%
   \COMMENT{}
    We shall refer to these%
   \COMMENT{}
    as the {\em projections\/} of $\vsq$, $\vs$ and $O$ to~$p$. Note that $\vS\restricts p = \vSq\restricts p \sub \vSp$ for all $p$ and $q\ge p$, with equality if our inverse system~$\S$ is surjective.

Even if our inverse system $(\,\vSp\mid p\in P\,) $ is not surjective, it induces the surjective inverse system $(\,\vS\restricts p\mid p\in P\,) $, whose inverse limit is again~$ \vS $ with the same topology. When we consider a given profinite separation system in this paper and need to work with a concrete inverse system of which it is the inverse limit, we can therefore always choose this inverse system to be surjective.

By Lemma~\ref{CompactnessThm}, our separation system $\vSinf=\varprojlim\S$ is a compact topological space. Every subset $O$ of $\vS$ induces a restriction of~$\S$, the inverse system $(\,O\restricts p\mid p\in P\,)$. While clearly $O\sub \invlim (\,O\restricts p\mid p\in P\,)$,%
   \COMMENT{}   
   this inclusion can be strict. In fact, $\invlim\! (\,O\restricts p\mid p\in P\,)$ is precisely the topological closure $\overline O$ of $O$ in~$\vS$:\looseness=-1

\begin{LEM}\label{closed}
\begin{enumerate}[\rm (i)]\itemsep=0pt
   \item The topological closure in~$\vS$ of a set $O\sub\vSinf$ is the set of all limits $\vs = (\,\vsp\mid p\in P\,)$ with $\vsp\in O\restricts p$ for all~$p$.
   \item A set $O\sub\vS$ is closed in~$\vSinf$ if and only if there are sets $O_p\sub\vSp$, with $f_{qp}\restricts O_q\sub O_p$ whenever $p<q\in P$, such that $O =%
   \COMMENT{}
   \varprojlim\, (\,O_p \mid p\in P\,)$.%
   \COMMENT{}
   \end{enumerate}
   \end{LEM}

\begin{proof}
Assertion (i) is straightforward from the definition of the product topology on~$\vS$. The forward implication of~(ii) follows from~(i) with $O_p = O\restricts p$. For the other direction,%
   \COMMENT{}
   note that $\varprojlim\, (\,O_p \mid p\in P\,)$ is a compact subspace of the Hausdorff space~$\vS$ (Lemma~\ref{CompactnessThm}).
\end{proof}

An immediate consequence of Lemma~\ref{closed} is that every separation system ${\tau\subseteq\vS}$ that is closed in $ \vS $ is itself profinite and can therefore be studied independently of its ambient system $ \vS $. We shall later have to deal in particular with nested subsystems of given separation systems. We shall therefore study profinite nested separation systems and tree sets first (Section~\ref{sec:limtreesets}).

Lemma~\ref{closed} will be our main tool for checking whether orientations of separation systems are closed. Let us illustrate how this is typically done.

\begin{EX}\label{ex:ray}
Consider a ray $ G=v_1e_1v_2e_2v_3\dots$, with end~$ \omega $ say, and let $ \vR$ be the separation system consisting of all separations of order~$ <k $ of the graph $ G $ for some fixed integer $k\ge 2$. Let $ P $ be the set of all finite subsets of $ V=V(G) $, and for $ p\in P $ let $ \vSp $ be the set of separations of order~$ <k $ of the subgraph~$ G[p] $. With restriction as bonding maps, this makes $ \S=(\vSp\mid p\in P) $ into an inverse system, with $\vR\simeq\vS = \invlim\!\S$ via $ (A,B)\mapsto (\,(A\cap p\,,\,B\cap p)\mid p\in P\,) $ for all $ (A,B) \in\vR$, as earlier.

Now let $ O $ be the orientation of $R$ `towards~$ \omega $', the set of all separations $(A,B)\in\vR $ such that $G[B]$ contains a tail of~$G$. Note that since $k$ is finite, $O$~is well defined and consistent. Clearly, $(V,\es)\notin O$.%
   \COMMENT{}
   It is also straightforward to check from first principles that $(V,\es)$ lies in the closure of~$O$.%
   \COMMENT{}

Instead, let us use Lemma~\ref{closed} to prove this. We just have to find separations $\vsp\in O\restricts p$ such that $\vs = (\,\vsp\mid p\in P\,)$ corresponds to~$(V,\es)$ under our isomorphism between $\vR$ and~$\vS$. Let $\vsp:= (p,\es)$ for each $p\in P$. These $\vsp$ lie in~$O\restricts p$, because $\vsp = \vr\restricts p$ for $\vr = (p, V\sm p)\in O$. But we also have $\vsp = (V,\es)\restricts p$ for each~$p$, so $(V,\es)$ corresponds to $\vs = (\,\vsp\mid p\in P\,)$ as desired.\qed%
   \COMMENT{}
\end{EX} 

In Example~\ref{ex:ray} we expressed $ (V,\emptyset) $, via the isomorphism $\vR\to\vS$, as the limit $\vs = (\,\vsp\mid p\in P\,)$ of its projections $\vsp = (p,\es)$. It can also be expressed as a limit in $\vR$ itself in various ways. For example, $(V,\es)$ is the supremum in~$\vR$ of the chain $ C=\{\,(A_n,B_n)\mid n\in\N\,\} \sub O$ where $A_n := \{\,v_i\mid i \le n\,\}$ and $B_n := \{\,v_i\mid i\ge n\,\}$. Topologically, $(V,\es)$~lies in the closure in~$ \vS $ not only of~$O$ but also of~$ C$.

Our next lemma generalizes these observations to our abstract~$\vS$: chains in~$ \vS $ always have a supremum, which lies in their topological closure.

\begin{LEM}\label{lem:chains}
If $ C\sub\vS $ is a non-empty chain, then $ C $ has a supremum and an infimum in~$ \vS $. Both these lie in the closure of $ C $ in~$ \vS $.
\end{LEM}

\begin{proof}
We prove the assertion about suprema; the proof for infima is analogous.

For each $ p\in P $, let $\vsp$ be the maximum of the finite chain $ C\restricts p $ in~$\vSp$. If these $\vsp$ are compatible, in the sense that $ \vs := (\,\vsp\mid p\in P\,)\in\vS$,%
   \COMMENT{}
   then $\vs$ is clearly the supremum of~$C$,%
   \COMMENT{}
   and it will lie in its closure by Lemma~\ref{closed}.

To show that the $\vsp$ are compatible, suppose that $ \vsq\restricts p\ne\vsp $ for some $p<q$. Then $ \vsq\restricts p<\vsp $, since $\vsq\restricts p$ lies in $(C\restricts q)\restricts p = C\restricts p$, whose maximum is~$\vsp$. Pick any $ \vr\in C $ with $ \vr\restricts p=\vsp $. Then $ \vr\restricts q\ne\vsq $ and hence $ \vr\restricts q<\vsq $, while $f_{qp} (\vr\restricts q) = \vr\restricts p = \vsp > \vsq\restricts p =  f_{qp} (\vsq)$. Hence $ f_{qp} $ does not respect the ordering on~$\vSq$, which contradicts our assumption that it is a homomorphism of separation systems.
\end{proof}

A direct consequence of Lemma~\ref{lem:chains} is that, in closed subsets of $ \vS $, every separation lies below some maximal element and above some minimal element:

\begin{LEM}\label{lem:closedsplitting}
If $O\sub\vSinf$ is closed in~$\vSinf$, then for every $\vs\in O$ there exist in~$O$ some $\vr\le\vs$ that is minimal in~$O$ and some $\vt\ge\vs$ that is maximal in~$O$.
\end{LEM}

\begin{proof}
Let $ C $ be a maximal chain in $ O $ containing~$ \vs $.%
   \COMMENT{}
   By Lemma~\ref{lem:chains}, $ C $~has an infimum~$\vr$ and a supremum $ \vt $ in $ \vS $. By the lemma these lie in~$ \overline{C} $, and hence in~$ O $ since $ O $ is closed. Then $ \vr $ is minimal and $\vt$ is maximal in $ O $, by the maximality of~$ C $, and $ \vr\le\vs\le\vt $ as desired.
\end{proof}

Recall that we shall be interested in tree sets%
   \COMMENT{}
   $\tau\sub\vSinf$, and in particular in their consistent orientations. As we noted earlier, the consistent orientations~$O$ of~$\tau$ can be retrieved as $O = \dcl(\sigma)$ from just their sets $\sigma$ of maximal elements, the splitting stars of~$\tau$---but only if every $\vr\in O$ lies below some maximal element of~$O$.
Lemma~\ref{lem:closedsplitting} says that this is the case when $O$ is closed in~$\vS$:%
   \COMMENT{}

\begin{COR}\label{closedsplits}
A consistent orientation of a nested separation system $\tau\sub\vS$ splits if it is closed in~$\vS$.\qed%
   \COMMENT{}
\end{COR}

In Section~\ref{sec:orientations} we shall see that Lemma~\ref{closedsplits} does not have a direct converse: profinite tree sets can have consistent orientations that split but are not closed in these tree sets.%
   \COMMENT{}

\section{Profinite nested separation systems}\label{sec:limtreesets}

Let us again consider a profinite separation system $\vS = \varprojlim\,(\,\vSp\mid p\in P\,)\ne\es$ with bonding maps $f_{qp}$, fixed throughout this section. Our aim in this section is to study the relationship between $ \vS $ and the projections~$ \vSp $. This will enable us later to lift theorems about finite separation systems to profinite ones.

We are particularly interested in the cases where the $ \vSp $ (and hence $ \vS $, as we shall see) are nested; we shall then study the interplay between the splitting stars of $ \vS $ and those of the~$ \vSp $. Understanding this interplay will be crucial to lifting tangle duality or tangle-tree theorems about finite separation systems to profinite ones, as both these types of theorem are about tree sets and their splitting stars.

\COMMENT{}



We begin our study with some basic observations. Since the $\vSp$ form an inverse system of separations, and $\vS$ is its inverse limit, we have
\begin{equation}\label{*}
\vr\le\vs\ \Rightarrow\ \vr\restricts p \le \vs\restricts p
\end{equation}
for all $\vr,\vs\in\vSq$ with $q\ge p$ or $\vr,\vs\in\vS$. In particular, projections of nested sets of separations are nested. Thus if our inverse system $(\,\vSp\mid p\in P\,)$ is surjective and $\vS$ is nested, then so are all the~$\vSp$. Projections of stars are stars, as long as they contain no degenerate separation (which are not allowed in stars).%
   \COMMENT{}

Projections also commute with inversion: $\sv\!\restricts p = (\vs\!\restricts p)^*$ for all $\!\vs$ as~above.%
   \COMMENT{}
   Hence projections of small separations are small, and projections of trivial separations are trivial if $p$ is large enough to distinguish them from their witness:

\begin{LEM}
Let $\vr = (\,\vrp\mid p\in P\,)\in\vS$ and $\vs = (\,\vsp\mid p\in P\,)\in\vS$ be such that $\vr$ is trivial in~$\vS$ and $s$ witnesses this. Then there exists $p\in P$ such that, for all $q\ge p$ in~$P$, the separation $\vrq$ is trivial in~$\vSq$ and $s_q$ witnesses this.
\end{LEM}

\begin{proof}
Since $s$ witnesses the triviality of~$\vr$, we have $\vr\notin\{\vs,\sv\}$. Hence there exists $p\in P$ such that $\vrq\notin\{\vsq, \svq\}$ for all $q\ge p$. For all such~$q$ we have $\vrq < \vsq$ as well as $\vrq < \sv\restricts q = \svq$ by~\eqref{*}, so $s_q$ witnesses the triviality of~$\vrq$ in~$\vSq$.
\end{proof}

Locally, for $p<q$, it can happen that trivial separations of~$\vSq$ project to nontrivial ones in~$\vSp$ if the witnesses $s_q$ of the trivialiy of $\vrq$ in~$\vSq$ project to the same separation in~$\vSp$ as $r_q$ does.%
   \COMMENT{}
   However, if $f_{qp}$ is onto, then nontrivial separations in~$\vSp$ will always also have a nontrivial preimage in~$\vSq$:

\begin{LEM}\label{LemmaD}
If $f_{qp}\colon \vSq\to\vSp$ is surjective, then for every $\vrp\in\vSp$ that is not trivial in~$\vSp$ there exists $\vrq\in\vSq$ that is not trivial in~$\vSq$ and such that $\vrq\proj p = \vrp$.
\end{LEM}

\begin{proof}
Let $\vrq$ be maximal in~$\vSq$ with $\vrq\proj p = \vrp$. Suppose $\vrq$ is trivial in~$\vSq$, witnessed by $s_q$ say. Then $\vrp\notin\{\vsq\proj p, \svq\proj p\}$, by the maximality of~$\vrq$. Hence $s\proj p\ne r_p$, and by~\eqref{*} it witnesses the triviality of~$\vrp$ in~$\vSp$.
\end{proof}

Let us see to what extent these observations, starting with~\eqref{*}, have a converse. Suppose we know that the projections in $\vSp$ of some elements of~$\vS$ or of $\vSq$ with $p<q$ are small, trivial, nested  or splitting stars: can we infer that these elements themselves are small, trivial, nested or splitting stars in $\vS$ or~$\vSq$?

The direct converse of~\eqref{*} will usually fail: two separations of a set, for example, can happily cross even if their restrictions to some small subset are nested. If $\vS$ is nested, however, we have a kind of converse of~\eqref{*}:

\begin{LEM}\label{*ii}
Let $\vr,\vs$ be elements of a nested set $\tau\sub\vS$, and let $p\in P$. Assume that $r\restricts p \ne s\restricts p$, that neither $\vr\restricts p$ nor $\sv\restricts p$ is trivial in~$\tau\restricts p$, and that $\vr\restricts p\le\vs\restricts p$. Then $r\ne s$ and $\vr\le\vs$.
\end{LEM}

\begin{proof}
Our assumption of $r\restricts p \ne s\restricts p$ implies $r = \{\vr,\rv\}\ne \{\vs,\sv\} = s$ by~\eqref{*}.%
   \COMMENT{}
   If orientations of $r$ and~$s$ were related in any way other than as $\vr\le\vs$ then, by~\eqref{*}, their projections to~$p$ would be related correspondingly. This would contradict our assumption of $\vr\restricts p \le\vs\restricts p$ as either $ \vr\restricts p $ or $ \sv\restricts p $ would then be trivial in $ \tau\restricts p $, or $ \vr\restricts p=\vs\restricts p $. Since $\tau$ is nested, however, $r$~and $s$ must have comparable orientations. The only possibility left, then, is that $\vr\le\vs$ as desired.
\end{proof}
   
Lemma~\ref{*ii} says that any ordering between two fixed projections of limits $\vr,\vs\in\vS$ lifts to $\vr$ and $\vs$ themselves if we assume that these are nested. Our next lemma says that this will be the case as soon as the $\vSp$ are nested:

\begin{LEM}\label{lemma3}
If every~$\vSp$ is nested, then so is~$\vS$.
\end{LEM}

\begin{proof}
If $r,s\in S$ are not nested, then $\vr\not\le\vs$ for all four choices of orientations $\vr$ of~$r$ and $\vs$ of~$s$. Each of these four statement is witnessed by some $p\in P$, in that $\vr\restricts p\not\le \vs\restricts p$. Since $P$ is a directed poset, there exists some $q\in P$ that is larger than these four~$p$. By~\eqref{*}, then, $\vr\restricts q\not\le\vs\restricts q$ for all choices of orientations $\vr$ of~$r$ and $\vs$ of~$s$. But these $\vr\restricts q$ and $\vs\restricts q$ are the four orientations of $r\restricts q$ and~$s\restricts q$, since the involution in~$\vS$ commutes with that in~$\vSq$ via the projection $\vS\to\vSq$. Hence $\vSq$ is not nested, contrary to assumption.
\end{proof}

The fact that projections of small separations are small also has a converse at the limit:

\begin{LEM}\label{small}
If $\vr = (\,\vrp\mid p\in P\,)\in\vS$ is such that each $\vrp$ is small, then $\vr$ is small.
\end{LEM}

\begin{proof}
As projections commute with the involutions~${}^*$, we have $\rv = {(\,\rvp\mid p\in P\,)}$. Hence if $\vrp\le\rvp$ for all $p\in P$, then $\vr\le\rv$ by definition of $\le$ on~$\vS$.
\end{proof}

Lemma~\ref{small} says that being small is a property that cannot disappear suddenly at the limit. But it is not clear that regular profinite separation systems, those without small elements, are inverse limits of finite regular separation systems. Such finite systems can, however, be found:

\begin{PROP}\label{prop:regular}
	{\it
		\tn{ }
   \vskip-\medskipamount\vskip0pt
		\begin{enumerate}\itemsep=0pt
			\item[\tn{(i)}] Every inverse limit of finite regular separation systems is regular.
			\item[\tn{(ii)}] Every profinite regular separation system is an inverse limit of finite regular separation systems.
		\end{enumerate}
	}
\end{PROP}

\begin{proof}
	Assertion (i) follows immediately from~\eqref{*}.
	
	For (ii) assume that $\vS = \invlim(\,\vSp\mid p\in P\,)$ is regular.%
   \COMMENT{}
   If some $ p_0\in P $ exists for which $ \vSpnought $ is regular then, by~\eqref{*}, every $ \vSp $ with $ p\ge p_0 $ is regular, and $ \vS=\invlim (\,\vSp\mid p\ge p_0\,) $ as desired.

Suppose now that no such $ p_0 $ exists, i.e., that no $ \vSp $ is regular. For each $ p\in P $ let $ X_p\ne\es $ be the set of small separations in $ \vSp $. Then $(\,X_p\mid p\in P\,) $ is a restriction of $(\,\vSp\mid p\in P\,) $, since $f_{qp}(X_q)\sub X_p$ for all $p<q$ by~\eqref{*}.%
   \COMMENT{}
   By Lemma~\ref{CompactnessThm} it has a non-empty inverse limit~$X$. Pick $ \vx\in X $. Then also $ \vx \in \vS$, and by construction every projection of $ \vx $ is small. But then Lemma~\ref{small} implies that $ \vx $ is small, contrary to our assumption that $ \vS $ is regular.
\end{proof}

Interestingly, we cannot replace the word `small' in Lemma~\ref{small} with `trivial'. It is true that if all projections of $\vr\in\vS$ are trivial then there exists $s\in S$ such that both $\vr\le\vs$ and $\vr\le\sv$. But this $s$ cannot, in general, be chosen distinct from~$r$: it may happen that $\vs=\vr$ is the only choice. In that case, we have said no more than that $\vr$ is small---which we know already from Lemma~\ref{small}.

Let us call $\vr = (\,\vrp\mid p\in P\,)\in\vS$ {\em finitely trivial\/} in~$\vS$%
   \COMMENT{}
   if for every $p\in P$ there exists $q\ge p$ in~$P$ such that $\vrq$~is trivial in~$\vS\restricts q$.%
   \COMMENT{}
  By Lemma~\ref{small}, any such~$\vr$ will be small, since if $\vrq$ is trivial and hence small, then $\vrp$ too will be small, by~\eqref{*}, if $p<q$. However, even if $\vS$ is a tree set and $ \vr $ is maximal in a closed consistent orientation of $ \vS $ it can happen that $\vr\in\vS$ is finitely trivial (but not trivial),%
   \COMMENT{}
   even with $\vrp$ trivial in~$\vS\restricts p$ for all $p\in P$. Here is an example:

\begin{EX}\label{ex:trivialproj}
Let $ P=\{1,2,\dots\}$. For each $ p\in P $ let $ \sigma_p $ be a proper star with $ p $~elements, and~$ \vSp $ the separation system consisting of $ \sigma_p $ and the respective inverses as well as a separation $ s_p=\{\vsp,\svp\} $, where $ \vsp $ is trivial in $ \vSp $ with exactly one element~$ \vtp $ of~$ \sigma_p $ witnessing this. Let $ f_{p+1,p}\colon\vSp{}_{+1}\to\vSp $ be the homomorphism which maps~$\vsp{}_{+1} $ and $\vtp{}_{+1} $ to $ \vsp $, maps $ \sigma_{p+1}\sm\{\vtp{}_{+1}\}$ bijectively to~$ \sigma_p $, and is defined on the inverses of these separations so as to commute with the inversion. These maps induce bonding maps $ f_{qp}\colon\vSq\to\vSp $ by concatenation.

Let $\vS:=\invlim (\vSp\mid p\in P) $. Then $ \vs:=(\,\vsp\mid p\in P\,)\in\vS $ is the only sep-\penalty-200\-ara\-tion in $ \vS $ whose projections to $ p $ do not meet $ \sigma_p $, for any $ p\in P $. It is easy to see that every $r\in S\sm\{s\}$ has an orientation $\vr$ such that $\vr\proj p = \vtp$ for some~$p$.%
   \COMMENT{}
   Then $\vr\proj p' = \vspdash$ for any $p'<p$, and $\vr\proj q\in \sigma_q\sm\{\vtq\}$ for all~$q>p$. So this $p$ is in fact unique given~$r$, and conversely $\vr$ is the only element of~$\vS$ such that $\vr\proj p = \vtp$; let us denote this~$\vr$ by~$\vrp$.

We have shown that $ \vS $ consists of the star $ \sigma:=\{\vs\}\cup\{\,\vrp\mid p\in P\,\} $ plus inverses. By construction $ \vs $ is finitely trivial; in fact, all of its projections are trivial. However $ \vs $ is not trivial in $ \vS $, because no $r_p$ could witness this: as $\vrp\proj q\in \sigma_q\sm\{\vtq\}$ for $q>p$, and $\vtq$ is the only element of $\sigma_q$ witnessing the triviality of~$\vsq$, the fact that $\vsq\le\vtq\le \rvp\proj q$ (since $\sigma_q$ is a star containing both $\vtq$ and~$\vrp\proj q$) implies that $\vsq\not\le \vrp\proj q$ and hence $\vs\not\le\vrp$ by~\eqref{*}.

Thus, $ \vs $ is a maximal element of the star $ \sigma $, which is a closed consistent orientation of $ \vS $.\qed
\end{EX}

Note how the separation $\vs$ in Example~\ref{ex:trivialproj} escapes being trivial even though all its projections are trivial, because the witnesses~$\vtp$ of the triviality of those projections do not lift to a common element of~$\vS$, which would then witness the triviality of~$\vs$. Rather, the $\vtp$ are projections of pairwise different separations $\vrp\in\vS$, none of which can alone witness the triviality of~$\vs$.

The fact that $\sigma$ can be a splitting star containing a finitely trivial separation, even though splitting stars cannot contain trivial separations (Lemma~\ref{Remark8}), thus seems to rest on the fact that $\sigma$ is infinite. This is indeed the case:%
   \COMMENT{}

\begin{LEM}\label{triviallimitsingleton}
	Assume that $ \vS $ is nested, and let $ \sigma=\{\vr,\vsone,\dots,\vsn\} $ be any finite splitting star of $ \vS $ (where $n\ge 0$).
   \vskip-\medskipamount\vskip0pt
	\begin{enumerate}\itemsep=0pt
	\item[\tn{(i)}] No element of $ \sigma $ is finitely trivial in~$ \vS $.
	\item[\tn{(ii)}] If $ \sigma = \{\vr\}$, then $ \vr\proj p $ is nontrivial in~$ \vS\proj p $ for every $p\in P$.
	\end{enumerate}
\end{LEM}

\begin{proof}
Let $O$ be a consistent orientation of $\vS$ that splits at~$\sigma$.

For (i) suppose without loss of generality that $ \vr $ is finitely trivial in $ \vS $. Let $ P'\sub P $ be the set of all $ p\in P $ for which $ \vr\proj p $ is trivial in $ \vS\proj p $, and for each $ p\in P' $ pick some $ s(p)\in S $ such that $ s(p)\proj p $ witnesses the triviality of $ \vr\proj p $ in~$ \vS\proj p $. Let $ \vs(p) $ be the orientation of $ s(p) $ that lies in $ O $.

As $ \vr\proj p<\vs(p)\proj p $ by the choice of~$ s(p) $ we cannot have $ \vs(p)\le\vr $, by~\eqref{*}. Hence $ \vs(p)\le\vsi $ for some~$i$. As~$ \sigma $ is finite, there is some $ i\in\{1,\dots,n\} $ for which the set $ P'' $ of all $ p\in P' $ with $ \vs(p)\le\vsi $ is cofinal in $ P' $ (and hence in $ P $). But now we have $ \vr\proj p<\vs(p)\proj p\le\vsi\proj p $ for every $ p\in P'' $, and hence $\vr\proj p\le\vsi\proj p$ for all $p\in P$.%
   \COMMENT{}
   But this means that $ \vr\le\vsi $, which contradicts the fact that $ \sigma $ is a proper star (Lemma~\ref{Remark8}).

For (ii) let $ p\in P $ be given and suppose that $ \vr\proj p $ is trivial in $ \vS\proj p $. Let $ \vs\in O $ be such that $ s\proj p $ witnesses the triviality of $ \vr\proj p $ in~$ \vS\proj p $. Then $\vr\proj p < \vs\proj p$. But as $ \{\vr\} $ splits~$ \vS $, witnessed by $O\owns \vs$, we also have $ \vs\le\vr $ and hence $ \vs\proj p\le\vr\proj p $, by~\eqref{*}. This is a contradiction.
\end{proof}

Our situation when lifting theorems about finite separation systems to profinite ones will be that we know the~$\vSp$ and wish to study~$\vS$. In particular, we shall be interested in its closed consistent orientations, given those of the~$\vSp$.

Let us begin by addressing a local question in this context. Suppose an orientation $O$ of $\vS$ that we would like to be consistent is in fact inconsistent, and that this is witnessed by an inconsistent pair $\{\sv,\vsdash\}\sub\vS$. Will this show in the projections of $\svp$ and~$\vsdashp$? More specifically, if we start with a consistent orientation $O_p$ of~$S_p$ (without degenerate elements) and take as $O$ the unique orientation of~$S$ that projects to a subset of~$O_p$,%
   \COMMENT{}
   will $O$ be consistent?

Given an inconsistent pair $\{\sv,\vsdash\}\sub\vS$, consider $\rvp := \sv\restricts p$ and $\vrdashp := \vsdash\restricts p$ for some $p\in P$. If $r_p\ne r'_p$ then $\{\rvp,\vrdashp\}$ is also inconsistent, by~\eqref{*}. But it can also happen that $r_p=r'_p$. In that case we would expect that $\rvp$ and $\vrdashp$ are inverse to each other, i.e.\ that $\vrp = \vrdashp$. This is indeed what usually happens, and it will not cause us any problems. (In our example, only one of $\rvp$ and $\vrdashp$ lies in~$O_p$, so only one of $\sv$ and~$\vsdash$ will be adopted for~$O$.)
   However it can also happen that\,$\sv$ and~$\vsdash$, despite being inconsistent, induce the same separation $\rvp = \vrdashp$ in~$\vSp\,$:

\begin{EX}\label{inconsistentpair}
Let $V\!$ be a disjoint union of four sets~$A,B,C,X$. Let $\vS$ be the system of set separations of~$V\!$, and consider its projections to $U:= C\cup X$. Let $\vs =  (X\cup A, C\cup X\cup B)$ and $\vsdash = (C\cup X\cup A, X\cup B)$. Then $\vs < \vsdash$, so $\{\sv,\vsdash\}$ is inconsistent, but their projections $\rvp = \vrdashp = (C\cup X,X)$ to~$U$ coincide.
\end{EX}

In fact, it can happen that {\em every\/} projection $\rv\restricts p$ of some $\rv\in\vS$ coincides with the projections to~$\vSp$ of some two inconsistent separations in~$\vS$ (depending on~$p$). Such curious separations~$\vr$ are rare, of course, and in practice they do not cause any big problems. For example, one can show that they are small, because all their projections~$\vr\proj p$ are~-- like $\vrp$ in the above example.%
   \COMMENT{}

Let us return to the question of what we can deduce about~$\vS$ from information we have about the~$\vSp$. We saw in Lemma~\ref{lemma3} that if all the $\vSp$ are nested then so is~$\vS$. Then the closed consistent orientations of~$\vS$, which we shall be interested in most, are essentially determined by its splitting stars (see Section~\ref{sec:orientations}).%
   \COMMENT{}
   Ideally, then, the stars~$\sigma$ splitting~$\vS$ should be the limits (as sets) of the stars splitting the~$\vSp$.%
   \COMMENT{}

Before we prove two lemmas saying essentially this, let us see why it cannot quite be true as stated. For a start, $\sigma$~might consist of two separations $\vsone,\vstwo\in\vS$ whose projections $\vsone\restricts p$ and $\vstwo\restricts p$ to some~$p\in P$ are inverse to each other (i.e., are the two orientations of the same unoriented separation in~$S_p$), and hence cannot lie in the same consistent orientation of~$S_p$.%
   \COMMENT{}
   More fundamentally recall that, by Lemma~\ref{Remark8}, splitting stars contain no trivial separations. But separations $\vs\in\vS$ that are not trivial in~$\vS$ may well have projections $\vs\restricts p$ that are trivial in~$\vSp$, as we saw in Example~\ref{ex:trivialproj}. Hence we cannot expect $\sigma\restricts p$ to split~$\vSp$ just because $\sigma$ splits~$\vS$.%
   \COMMENT{}

Our next lemma shows that these two examples are essentially the only ones.%
   \COMMENT{}
   The first can be overcome by choosing $p$ large enough, while in the second we only have to delete the trivial separations in $\sigma\restricts p$ to obtain a splitting star of~$\vSp$.

Recall that $\vS{}^\circ$ denotes the {\em essential core\/} of~$\vS$, the separation system obtained from~$\vS$ by deleting its degenerate, trivial, and co-trivial elements.%
   \COMMENT{}
   In the lemma below,%
   \COMMENT{}
   let $\amgis_p := \{\sv\mid\vs\in\sigma_p\}$.

\begin{LEM}\label{lemma5}
Assume that $ (\,\vSp\mid p\in P\,) $ is surjective, and that the~$\vSp$ and hence~$\vS$%
   \COMMENT{}
   are nested. Assume further that $ \sigma\sub\vS $ splits $ \vS $, witnessed by the consistent orientation $ O $ of $ \vS $. Then either $ \sigma=\{\rv\} $ with $ \vr $ finitely trivial in~$ \vS $ or $r$ degenerate, or there exists $ p_0\in P $ such that, for all $ p\ge p_0 $, the set
  $$ \sigma_p:=(\sigma\proj p)\cap\vSp{\!\!}^\circ $$
  splits~$ \vSp $, witnessed by the consistent orientation~$ O_p:=(O\proj p)\sm\amgis_p $ of~$ \vSp $.%
   \COMMENT{}
\end{LEM}

\begin{proof}
Consider first the case that $ \abs{\sigma}=1 $, i.e., that $ \sigma $ has the form $ \{\vr\} $ for some $ \vr=(\,\vrp\mid p\in P\,)\in\vS$. By Lemma~\ref{triviallimitsingleton}, $ \vr $~cannot be finitely trivial, and if $ \rv $ is finitely trivial we are done by assumption. Thus we may assume that neither $ \vr $ nor $ \rv $ is finitely trivial in~$ \vS $, and that $r$ is not degenerate.%
   \COMMENT{}
   Then there exists $ p_0\in P $ such that $ \vrp\in\vSp{\!\!}^\circ $ for all $ p\ge p_0 $.%
  \COMMENT{} 
  We claim that this $ p_0 $ is as desired. 

To see this, observe that $ \vr $ is the greatest element of $ O $, which by~\eqref{*} implies that $ \vrp $ is the greatest element of $ O_p $. Let us show that $ O_p $ is an orientation of~$ \vSp $. Suppose there is a nondegenerate $ s\in\vSp $ for which $ \vs,\sv\in O_p $. Then $ \vs<\vrp $ and $ \sv<\vrp $, so $ \vrp $ is co-trivial in $ \vSp $, contrary to the choice of $ p_0 $.%
   \COMMENT{}
   A~similar argument shows that $O_p$ is consistent.%
   \COMMENT{}

Let us now consider the case that $ \abs{\sigma}\ge 2 $. We shall first show that there exists some $ p_0\in P $ such that all $ O\proj p $ with $ p\ge p_0 $ are consistent orientations of~$ \vSp $. We then show that this $ p_0 $ is as desired.

By the surjectivity of the bonding maps we have $ \vSp=\vS\proj p $ for each $ p\in P $,%
   \COMMENT{}
   so $ O\proj p $ contains at least one of $ \vsp $ and $ \svp $ for each $ s_p\in\vSp $. It thus suffices to find $ p_0 $ so that $ O\proj p $ is antisymmetric and consistent for all $ p\ge p_0 $.

For every $ p\in P $ let 
 $$ I_p:=\{\,(\vrp,\vsp)\mid\vrp,\vsp\in O\proj p\tn{ and }\rvp\le\vsp\,\}.$$
    Note that $ O\proj p $ is a consistent orientation of $ \vSp $ if (and only if) $ I_p $ is empty: any pair $ (\vrp,\vsp) $ of separations in $ O\proj p $ that witnesses that $ O\proj p $ is inconsistent lies in~$ I_p $, and so does any pair witnessing that $ O\proj p $ is not antisymmetric. Furthermore, for any $ q>p $, if $ (\vrq,\vsq) $ lies in $ I_q $ then $ (\vrq\proj p\,,\,\vsq\proj p) $ lies in $ I_p $. Thus if~$ I_{p_0} $ is empty for some $ p_0\in P $ then $ I_p $ is empty for all $ p\ge p_0 $.%
    \COMMENT{}

Suppose $ I_p $ is non-empty for all $ p\in P $. The family $ (\,I_p\mid p\in P\,) $ forms an inverse system with bonding maps borrowed from $ (\,\vSp\mid p\in P\,) $, which by Lemma~\ref{CompactnessThm} has a non-empty inverse limit. Let $ (\vr,\vs)=(\,(\vrp,\vsp)\mid p\in P\,) $ be an element of this inverse limit. Then $ \vr=(\,\vrp\mid p\in P\,) $ and $ \vs=(\,\vsp\mid p\in P\,) $ are elements of~$ \vS $ satisfying $ \rv\le\vs $.%
   \COMMENT{}
   By Lemma~\ref{lem:boundedclosed},%
   \footnote{\dots in whose proof we shall not use Lemma~\ref{lemma5}}
  $ O $~is closed in~$ \vS $, giving $ \vr,\vs\in O $ by Lemma~\ref{closed}(i) since $\vrp,\vsp\in O\proj p$ for all~$p$.

But this contradicts the fact that $ O $ is a consistent orientation of $ \vS $. Indeed, if $ r\ne s $ then $ \vr $ and $ \vs $ witness its inconsistency. If $ r=s $ with $ \vr=\vs $, then $ \rv\le\vs=\vr $, so $ \vr $ is co-small. Choose $\vsdash\in\sigma$ so that $\vs\le\vsdash$. Then every $\vt\in\sigma\sm\{\vsdash\}$ (which exists, since $|\sigma|\ge 2$) is trivial, as $\vt < \svdash\le \sv\le\vs$, which contradicts Lemma~\ref{Remark8}. And finally, if $ r=s $ with $ \vr=\sv $, then $ r $ must be degenerate as $ O $ is antisymmetric, but in that case $ \sigma=\{\vr\} $ by Lemma~\ref{Remark8}, contradicting $ \abs{\sigma}\ge 2 $. This completes our proof that there exists $ p_0\in P $ such that $ O\proj p $ is a consistent orientation of $ \vSp $ for all $ p\ge p_0 $.

Let us now show that this $ p_0 $ is as claimed. For this we only need to check, for all~$ p\ge p_0 $, that $ \sigma_p:=(\sigma\proj p)\cap\vSp{\!\!}^\circ$ is the set of maximal elements of~$ O\proj p $: since $ O\proj p $ is a consistent orientation of~$\vSp$, by the choice of~$p_0$, and finite, it clearly splits at the set $\tau_p$ of its maximal elements.

To see that $ \tau_p $ is contained in $ \sigma_p $ let $ \vrp\in\tau_p $ be given and pick some $ \vr\in O $ with $ \vr\proj p=\vrp $. As $ O$ splits at~$ \sigma $, there is some $ \vs\in\sigma $ such that $ \vr\le\vs $. But then $ \vrp\le\vs\proj p $ by~\eqref{*}, and hence $ \vrp=\vs\proj p $ by the maximality of $ \vrp $ in $ O\proj p $. Thus, $ \vrp\in \sigma\proj p $. But this implies $ \vrp\in\sigma_p $, since $ \tau_p\sub\vSp{\!\!}^\circ $ by Lemma~\ref{Remark8}.

For the converse inclusion let $ \vsp\in\sigma_p $ be given and pick $ \vs\in\sigma $ with~${ \vs\proj p=\vsp} $. If $ \vsp $ is maximal in $ O\proj p $ then $ \vsp\in\tau_p $ and there is nothing to show. Otherwise there exists $ \vtp\in\tau_p $ with $ t_p\ne s_p $ and $ \vsp\le\vtp $. As seen above we can find $ \vt\in\sigma $ with $ \vt\proj p=\vtp $. Then $ \vs\ne\vt $ and hence $ \vs\le\tv $ by the star property, which implies $ \vsp\le\tvp $ by~\eqref{*}. But then $ t_p $ witnesses that $ \vsp $ is trivial in $ \vSp $ and hence does not lie in $ \sigma_p $, contrary to our assumption.
\end{proof}

In applications of Lemma~\ref{lemma5} it will be convenient to know that all the separations deleted from~$\sigma\proj p$ in the definition of $\sigma_p$ are in fact trivial, not co-trivial,%
   \COMMENT{}
   and even trivial in~$\sigma\proj p$ rather than just in the larger~$\vSp$. Our next lemma will imply this (with $\sigma\proj p$ as the star considered there).%
   \COMMENT{}

Given a subset $\sigma$ of a separation system~$\vS$, write $\sigma^-\,$for the set obtained from $\sigma$ by deleting any separations that are trivial in~$\sigma$ (that is, trivial in $\vS$ with a witess in~$\sigma$; see the start of Section~\ref{sec:separations} for a formal definition). Note that if $\sigma$ is finite,%
   \COMMENT{}
   then any $\vr\in\sigma\sm\sigma^-$ is not only trivial in~$\sigma$ but has a witness of this in~$\sigma^-$: just take a maximal witness in~$\sigma$, which cannot also be trivial in~$\sigma$ since its own witnesses would have an orientation that is a greater witness of the triviality of~$\vr$.

\begin{LEM}\label{LemmaB}
Let $\vS$ be a nested separation system split by a star~$\sigma^\circ$, and let $\sigma\sub\vS$ be any star containing~$\sigma^\circ\!$.
   \begin{enumerate}[\rm(i)]
   \item If $\sigma$ is antisymmetric, then $\sigma^\circ\! = \sigma\cap \vS{}^\circ = \sigma^-\!$, and any element of~$\sigma$ that is trivial in~$\vS$ has a (nontrivial) witness for this in~$\sigma^\circ\!$.%
   \COMMENT{}
   \item If $\sigma$ is not antisymmetric, then there exists $\vs\in\vS{}^\circ$ such that $\sigma^\circ = \{\vs\}$ and $\sigma^- = \{\vs,\sv\}$.%
   \COMMENT{}
   \end{enumerate}
\end{LEM}

\begin{proof}
(i) By Lemma~\ref{Remark8} we have $\sigma^\circ\sub\vS{}^\circ$. As $\sigma^\circ\sub\sigma$, this yields $\sigma^\circ\sub\sigma\cap \vS{}^\circ$.
And we have $\sigma\cap \vS{}^\circ \sub \sigma^-\!$, because elements of~$\sigma$ that are trivial in~$\sigma$ are also trivial in~$\vS$.%
   \COMMENT{}
   For the converse inclusions, and the second statement of~(i), we show that any $\vr\in \sigma\sm \sigma^\circ$ is trivial in $\sigma$ with a witness in~$\sigma^\circ\sub\sigma$, and hence does not lie in~$\sigma^-$ either.%
   \COMMENT{}

Let $O$ be the consistent orientation of~$\vS$ whose set of maximal elements is~$\sigma^\circ$. If $\rv\in O$, then $\rv\le\vs$ for some $\vs\in\sigma^\circ$, by the choice of~$O$. The inequality is strict, since $\rv\ne\vs$ as $\vr,\vs\in\sigma$ and $\sigma$ is antisymmetric by assumption. Hence $\rv < \vs\le\rv$ by the star property of~$\sigma\owns \vr,\vs$, with a contradiction.

Thus, $\rv\notin O$ and therefore $\vr\in O$. As before, there exists $\vs\in\sigma^\circ$ such that $\vr \le \vs$. This time, the inequality is strict because $\vs\in\sigma^\circ$ while $\vr\in\sigma\sm\sigma^\circ$. But we also have $\vr\le\sv$, by the star property of~$\sigma\owns \vr,\vs$. As $\sigma$ is antisymmetric by assumption, the inequality $\vr\le\sv$ is strict, too. Thus, $\vr$~is trivial in~$\vS$ with a witness $\vs\in\sigma^\circ$, as desired.

(ii) If $\sigma$ is not antisymmetric, then $\sigma\supe \{\vs,\sv\}$ for some~$s\in S$. Since $\sigma$ is a star, any other element of~$\sigma$ is trivial with witness~$s$. In particular, $\sigma^\circ\sub\sigma$ contains no such other elements by Lemma~\ref{Remark8}. As $\sigma^\circ$ is antisymmetric,%
   \COMMENT{}
   we thus have $\sigma^\circ = \{\vs\}$ for one of the orientations $\vs$ of~$s$.

As any element of $\sigma$ other than $\vs$ and~$\sv$ is trivial in~$\sigma$, we have ${\sigma^-\sub\{\vs,\sv\}}$. As $\vs\in\sigma^\circ\sub\vS{}^\circ$ is neither trivial nor co-trivial even in~$\vS$ we also have the converse inclusion, i.e., $\sigma^- = \{\vs,\sv\}$ as claimed.
\end{proof}

Next, let us prove a local converse of Lemma~\ref{lemma5}, which implies that every star splitting an element~$\vSp$ of a surjective inverse system $(\,\vSp\mid p\in P\,)$ of nested separation systems is induced, modulo the deletion of trivial separations,%
   \COMMENT{}
   by some splitting star in every~$\vSq$ with $q>p$.%
   \COMMENT{}
   Our proof will be independent of Lemma~\ref{LemmaB}, since using it would not make it much shorter.

\begin{LEM}\label{LemmaX}
Let $f_{qp}\colon \vSq\to\vSp$ be an epimorphism between two nested separation systems without degenerate elements, denoted as $\vs\mapsto\vs\proj p$.%
   \COMMENT{}
   If $\sigma_p\sub \vSp$ splits~$\vSp$, then $\vSq$ is split by a set $\sigma_q$ such that $(\sigma_q\restricts p)\cap \vSp{\!\!}^\circ = (\sigma_q\proj p)^- = \sigma_p$.
   \COMMENT{}
\end{LEM}

\begin{proof}
Let $O_p$ be a consistent orientation of~$\vSp$ of which $\sigma_p$ is the set of maximal elements. By Lemma~\ref{Remark8}, $\sigma_p$~is a star in~$\vSp{\!\!}^\circ$ (which it clearly also splits). Let
\[ O_q:=\{\vs\in\vSq\mid \vs\restricts p\in O_p\}\,. \]
As $ \vSp $ contains no degenerate element, $ O_q $~is antisymmetric: an orientation of~$\vSq$.%
   \COMMENT{}

Here our proof splits into two cases: that $\sigma_p$ has a co-small element or not. We first assume it does not.

Let us show that $ O_q $ is consistent. Suppose it is not. Then there are ${\rv,\vrdash\in O_q}$ such that $ \vr < \vrdash $ (and $ r\ne r' $). By~\eqref{*}%
   \COMMENT{}
   we have $ \vr\proj p\le\vrdash\proj p $. If ${r\proj p\ne r'\proj p}$ then this violates the consistency of $ O_p $, which contains $\rv\restricts p$ and $\vrdash\restricts p$ because ${\rv,\vrdash\in O_q}$. Hence $ r\proj p=r'\proj p $, with orientations $ \rv\proj p=\vrdash\proj p $ since $O_p$ contains both but is antisymmetric. Thus, $\vr\restricts p = \rvdash\restricts p$ is small, while $ \vrdash\proj p\in O_p$ lies below ($\le$) some $ \vs\in\sigma_p $. But then we must have $ \vrdash\proj p=\vs $ as otherwise $ \sv < \rvdash\restricts p\le\vrdash\restricts p$ is trivial, which contradicts our assumption that $\vs\in\sigma_p\sub \vSp{\!\!}^\circ$. Hence $ \vrdash\proj p=\vs \in\sigma_p$ is co-small, contradicting our assumption.

We have shown that $O_q$ is a consistent orientation of~$\vSq$. Let $ \sigma_q $ be the set of its maximal elements, a~splitting star of $ \vSq $. Let us show that $ (\sigma_q\proj p)\cap\vSp{\!\!}^\circ= (\sigma_q\proj p)^- =\sigma_p $.

Clearly, $(\sigma_q\proj p)\cap\vSp{\!\!}^\circ \sub (\sigma_q\proj p)^-$. For a proof of $(\sigma_q\proj p)^-\sub\sigma_p $ consider any $\vs\in\sigma_q$ for which $ \vs\proj p$ is not trivial in~$\sigma_q\proj p$.%
   \COMMENT{}
   As $\vs\in O_q$ we have $\vs\restricts p\in O_p$. Hence if $\vs\restricts p\notin\sigma_p$ then $\vs\restricts p < \vsdash\restricts p \in\sigma_p$ for some $\vsdash\in\vSq$. (Here we use that $f_{qp}$ is surjective.) Note that $\vsdash\in O_q$, by definition of~$O_q$. As $\vs\proj p$ and~$\vsdash\proj p$ both lie in~$O_p$, they cannot be inverse to each other, so $s\restricts p\ne s'\restricts p$. Recall that $\vs\restricts p$ is not trivial in~$\sigma_q\proj p$, by assumption. As $O_p$ is consistent, $\vsdash\restricts p\in O_p$ cannot be co-trivial either, even in~$\vSp$.%
   \COMMENT{}
   (Apply Lemma~\ref{lem:extension}(i) with $P=\{\vsdash\proj p\}$.)%
   \COMMENT{}
   Hence $\vs < \vsdash$, by Lemma~\ref{*ii} applied to $\sigma_q\cup\{\vsdash\}$ and $(\sigma_q\cup\{\vsdash\})\proj p = \sigma_q\proj p$.%
   \COMMENT{}
   This contradicts the maximality of~$\vs$ in~$O_q$ as an element of~$\sigma_q$.

To show the converse inclusion $ \sigma_q\proj p\supseteq\sigma_p $ (note that $ \sigma_p\sub\vSp{\!\!}^\circ $ by Lemma~\ref{Remark8}), consider any $ \vr\in\sigma_p $ and choose $ \vs\in O_q $ maximal with $ \vs\proj p=\vr $.%
   \COMMENT{}
   We need to show that $ \vs\in\sigma_q $, i.e., that $ \vs $ is maximal in $ O_q $. By \eqref{*}, any $\vsdash > \vs$ in~$O_q$ satisfies $\vsdash\restricts p\ge\vs\restricts p$. As $\vsdash\restricts p\in O_p$ by definition of~$O_q$, and $\vs\restricts p = \vr$ is maximal in~$O_p$ as an element of~$\sigma_p$, we must have equality: $\vsdash\restricts p = \vs\restricts p$. But this contradicts the choice of~$\vs$ in~$O_q$. Hence such an~$\vsdash$ does not exist, so $\vs\in\sigma_q$ as desired. This completes our proof that $(\sigma_q\restricts p)\cap\vSp{\!\!}^\circ  = (\sigma_q\proj p)^- = \sigma_p$, and hence of the lemma, in the case that $\sigma_p$ contains no co-small separation.

Let us now assume that $ \sigma_p $ does contain a co-small separation, $\vsp$~say.%
   \COMMENT{}
   Then the star $ \sigma_p $ cannot contain any other separations, as any such separation $\vr < \svp\le\vsp$ would be trivial and thus contradict Lemma~\ref{Remark8}.%
   \COMMENT{}
   
Let $ {M}\sub O_q $ be the set of all $ \vs \in \vSq $ with $ \vs\proj p=\vsp $. As before, for any inconsistent pair $ \rv,\vrdash\in O_q $ we have $ \rv\proj p=\vrdash\proj p\in \sigma_p = \{\vsp\}$ and this separation is co-small,%
   \COMMENT{}
   giving $ \rv,\vrdash\in{M} $.

Let us show that no minimal element $ \vrdash $ of $ {M} $ can be co-trivial in $ \vSq $. Any witness $ r $ for this would have orientations $ \rv,\vr <\vrdash $. One of them, $ \rv$~say, would be in~$ O_q $. Then $\rv,\vrdash$ are an inconsistent pair in~$O_q$, giving $ \rv,\vrdash\in{M} $ as above.%
   \COMMENT{}
   But now $ \rv<\vrdash $ contradicts the minimality of~$ \vrdash $ in~$M$.

Similarly, no $ \vr\in{M} $ can be trivial in~$\vSq$. Indeed, suppose $s\in S_q$ witnesses the triviality of~$\vr$. Then $\vr < \vs$ as well as $\vr < \sv$, which implies $\vr\proj p\le\vs\proj p$ as well as $\vr\proj p \le \sv\proj p$ by~\eqref{*}. One of these, $\vs\proj p$~say, lies in~$O_p$. Then $\vsp = \vr\proj p = \vs\proj p$ by the maximality of $\vsp$ in~$O_p$. But now $\vsp = \vs\proj p = \vr\proj p\le\sv\proj p$, so $\vsp$ is small as well as, by assumption, co-small, and hence degenerate. This contradicts our assumptions about~$\vSp$.

As $M$ has a minimal element,%
   \COMMENT{}
   we thus have $M':= M\cap\vSq{\!\!}^\circ\ne\es;$%
   \COMMENT{}
   let $ \vsone\in{M'}$ be maximal. Then $ \vsone $ is also maximal in $O'_q := O_q\sm (M\sm \{\vsone\})\sub O_q$: for any $ \vs\in O_q $ with $ \vsone\le\vs $ we have $\vsp = \vsone\proj p \le\vs\proj p\in O_p$, with equality by the maximality of~$ \vsp $ in~$O_p$, and hence $ \vs\in{M} $.%
   \COMMENT{}
   Now $O'_q$ is a partial orientation of~$S_q$, and it is consistent, since it has only one element in~$M$.%
   \COMMENT{}
   By Lemma~\ref{lem:extension}, $O'_q$~therefore extends to a (unique) consistent orientation $ O $ of $ \vSq $ in which $ \vsone $ is maximal. Let $ \sigma_q \owns \vsone$ be the set of all the maximal elements of~$O$. Then $\sigma_q$ is a splitting star of~$\vSq$. 

To complete our proof we have to show that $(\sigma_q\restricts p)\cap \vSp{\!\!}^\circ = (\sigma_q\proj p)^- = \sigma_p$. We have $\sigma_p=\{\vsp\}\sub (\sigma_q\proj p)\cap\vSp{\!\!}^\circ\sub (\sigma_q\proj p)^-$,%
   \COMMENT{}
   since $ \vsp=\vsone\proj p $ and $ \vsone\in\sigma_q $. For the converse inclusion let $ \vs\in\sigma_q $ be arbitrary; we show that if $ \vs\proj p\ne\vsp $ then $ \vs\proj p $ is trivial in~$\sigma_q\proj p$ and hence does not lie in~$(\sigma_q\proj p)^-$.

Assume that $\vs\proj p\ne\vsp$. Then $\vs\ne\vsone$. By~\eqref{*} and the star property for~$\sigma_q$ we therefore have $ \vs\proj p\le\svone\proj p=\svp\le\vsp $, so $ \vs\proj p $ is trivial unless $ \vs\proj p=\svp $. Then $ \sv\in{M} $ with $ \vsone\le\sv $.%
   \COMMENT{}
   The inequality is strict, since $\vsone$ and~$\vs$ are distinct%
   \COMMENT{}
   elements of the asymmetric set~$O$. Hence $\sv\notin M'$ by the choice of~$\vsone$, and thus $\sv\in M\sm M'$. Now $\sv$ cannot be trivial in~$\vSq$, since this would make $\vsone<\sv$ trivial too.%
   \COMMENT{}
   Hence $ \sv $ is co-trivial in $ \vSq $. But then $ \vs $ is trivial and thus cannot be maximal in~$ O $%
   \COMMENT{}
   (Lemma~\ref{lem:extension}(ii)), a contradiction.
   \end{proof}

Recall that our aim was, broadly, to show that the splitting stars $\sigma$ of $\vS$ are precisely the limits of the splitting stars of the projections~$\vSp$. In Lemma~\ref{lemma5} we showed that, except for the pathological case that $\sigma$ is a `finitely co-trivial' singleton, this is indeed true for all large enough~$p$.%
   \COMMENT{}
   In Lemma~\ref{LemmaX} we proved only a local converse of this: we did not show that the splitting stars of $\vSp$ are induced by splitting stars of~$\vS$, but by splitting stars of~$\vSq$ for all $q>p$. This is what we shall need in~\cite{duality1inf}.

Here, then, is a more direct converse of Lemma~\ref{lemma5}:

\begin{PROP}\label{prop:splittingpreimage}
Assume that $(\,\vSp\mid p\in P\,)$ is surjective, and that each of the $\vSp$ is nested and contains no degenerate separations. Let $p\in P$ be given. If $\sigma_p\sub \vSp$ splits~$\vSp$ and contains no co-small separation, then $\vS$ is split by a set~$\sigma$ such that $(\sigma\restricts p)\cap \vSp{\!\!}^\circ =(\sigma\proj p)^-= \sigma_p$.
\end{PROP}

\begin{proof}
Let $ O_p $ be the consistent orientation of $ \vSp $ witnessing that $ \sigma_p $ splits~$ \vSp $, and let $ O:=\{\vs\in\vS\mid\vs\proj p\in O_p\} $. As in the proof of Lemma~\ref{LemmaX}, first case, $O$~is a consistent orientation of~$\vS$. Also as before, its set $ \sigma $ of maximal elements satisfies $ (\sigma\proj p)\cap\vSp{\!\!}^\circ=(\sigma\proj p)^-=\sigma_p $. But in order to show that $\sigma$ splits~$\vS$ we must prove that every element of~$O$ lies below some element of~$\sigma $. This follows from Lemma~\ref{lem:closedsplitting} if $ O $ is closed in $ \vS $. But by Lemma~\ref{closed} and the definition of~$ O $, the topological closure of $ O $ in $ \vS $ is $ O $ itself. So $ O $ is indeed closed in $ \vS $.
\end{proof}

What happens in Proposition~\ref{prop:splittingpreimage} if $\sigma_p$ does contain a co-small separation, $\vsp$~say? Then $\sigma_p = \{\vsp\}$ as before. Suppose $\vS$ has a splitting set $ \sigma $ as desired, one such that $ (\sigma\proj p)\cap\vSp{\!\!}^\circ=(\sigma\proj p)^-=\sigma_p $, witnessed by the consistent orientation $O$ of~$\vS$, say. Then $ \sigma $~has an element $ \vs $ in $M:=\{\,\vs\in\vS \mid \vs\proj p=\vsp\, \}$. Note that $\vs\in\sigma$ is neither trivial nor co-trivial, by Lemma~\ref{Remark8}. Let us show that $ \vs $ is maximal in $M\cap\vS{}^\circ$, indeed in the set $M'$ of elements of $M$ that are not co-trivial in~$\vS$.%
   \COMMENT{}

Suppose $M'$ has an element $ \vsdash > \vs $. As $\vsdash\notin O$ by the maximality of $\vs$ in~$O$,%
   \COMMENT{}
   we have $\svdash\in O$. Then $\sigma$ has an element~$\tv$ such that $\svdash\le\tv$. If $ \tv=\vs $ then $\svdash\le\vs < \vsdash$. Since $\svdash$ is not trivial, this can happen only with equality $\svdash=\vs$ \cite[Lemma~2.4]{AbstractSepSys}. But then $\vs$ is small, and hence so is~$\vs\proj p = \vsp$, by~\eqref{*}. Since $\vsp$ is also co-small it is degenerate, a contradiction. Hence $ \tv\ne\vs $. As $\tv$ and~$\vs$ lie in the star~$\sigma$, we have $ \vs\le\vt $, as well as $ \vt\le\vsdash $ by choice of~$ \tv $. But then $\vs\le\vt\le\vsdash$ and thus, by~\eqref{*}, $ \vsp\le\vt\proj p\le\vsp$ with equality.%
   \COMMENT{}
   In particular, $\tv\proj p\in\vSp{\!\!}^\circ$,%
   \COMMENT{}
   and hence $\tv\proj p\in (\sigma\proj p)\cap\vSp{\!\!}^\circ=\sigma_p = \{\vsp\}$, since $\tv\in\sigma$. But now $\vt\proj p = \vsp = \tv\proj p$, contradicting our assumption that $\vSp$ has no degenerate elements. This completes the proof that $\vsdash$ does not exist, and hence that our arbitrary $\vs\in\sigma\cap M$ is maximal in~$M'\supe M\cap\vS{}^\circ $.

The upshot of all this is that if $\sigma_p$ contains a co-small separation, we can only hope to find a splitting star $ \sigma $ of~$\vS$ inducing~$\sigma_p$, in the sense that $(\sigma\proj p)\cap\vSp{\!\!}^\circ=(\sigma\proj p)^-=\sigma_p $, if the set~$ M\cap\vS{}^\circ $ has a maximal element. However, while $M$~is closed in $ \vS $ and thus has maximal elements above all its elements (Lemma~\ref{lem:closedsplitting}), this need not be the case for $M\cap\vS{}^\circ$.  Indeed, it is even possible that every $\vSp$ contains such a splitting singleton $\{\vsp\}$ for which the conclusion of Proposition~\ref{prop:splittingpreimage} fails; so we cannot even get this conclusion for sufficiently large $p\in P$.%
   \COMMENT{}

\medbreak

Let us complete this section with a special case of Lemma~\ref{LemmaB} that has a shorter proof and is of special interest: the case that the star $\sigma$ containing a splitting star~$\sigma^\circ$ of~$\vS$ is its closure. While the closure of an arbitrary subset of~$\vS$ may well add nontrivial elements to that set, this is not the case when the set is a splitting star:

\begin{PROP}\label{bonuslemma}
	Assume that $ \vS=\invlim\!(\,\vSp\mid p\in P\,)$ is nested and has no degenerate elements. Let $ \sigma^\circ\sub\vS $ be a splitting star of~$ \vS $, and let $ {\sigma}\supseteq\sigma^\circ $ be its topological closure in~$ \vS $. Then $ {\sigma} $ is a star, and $ {\sigma}^-\! =\sigma^\circ $.
\end{PROP}

\begin{proof}
	Let us show first that $ {\sigma} $ is a star. Given distinct $ \vr,\vs\in{\sigma} $, let $ p_0\in P $ be large enough that $ \vr\proj p_0\ne\vs\proj p_0 $. By Lemma~\ref{closed} there are for each $ p\ge p_0 $ distinct $ {\vrdash},{\vsdash}\in\sigma^\circ $ with $ \vr\proj p={\vrdash}\proj p $ and $ {\vsdash}\proj p=\vs\proj p $. As $ \sigma^\circ $ is a star we have $ {\vrdash}\le{\svdash} $, so $ \vr\proj p\le\sv\proj p $ by~\eqref{*}. This holds for every $ p\ge p_0 $ and hence shows that $ \vr\le\sv $ in~$ \vS $.
	
	It remains to show that $ {\sigma}^-\! = \sigma^\circ$. By Lemma~\ref{Remark8} we have $\sigma^-\!\supe\sigma^\circ$. To show th converse inclusion, let $ \vr\in{\sigma}\sm\sigma^\circ $ be given; we show that $\vr$ is trivial in~$\sigma$.

Suppose first that $ \rv\in\sigma^\circ $. As $\sigma$ is a star, any $ \vs\in\sigma^\circ $ with $ \vs\ne\rv $ is trivial with witness~$ r $, so $ \sigma^\circ=\{\rv\} $ by Lemma~\ref{Remark8}. But then $\vr\in {\sigma}=\sigma^\circ \not\owns\vr$, a~contradiction.%
   \COMMENT{}
	
	Suppose now that $ \rv\notin\sigma^\circ $. Let $ O $ be the orientation of $ \vS $ witnessing that $ \sigma^\circ $ splits $ \vS $. Then either $ \vr\in O $ or $ \rv\in O $. If $ \rv\in O $ there exists some $ \vs\in\sigma^\circ $ with $ \rv\le\vs $, as $ O $ splits at $ \sigma^\circ $. But also $ \vs\le\rv $ as $ {\sigma}\supseteq\sigma^\circ $ is a star, giving $ \rv=\vs\in\sigma^\circ $, contrary to assumption. Therefore we must have $ \vr\in O $. As $ O $ splits at $ \sigma^\circ $ there is an $ \vs\in\sigma^\circ $ with $ \vr < \vs$.%
   \COMMENT{}
   But we also have $ \vr < \sv $ by the star property of~${\sigma}$; note that $\vr\ne\sv$, since otherwise $\rv=\vs$, while $\rv\notin\sigma^\circ\owns\vs$ by assumption. Thus, $ \vr $ is indeed trivial with a witness in~$\sigma^\circ\sub\sigma$, as desired.
\end{proof}

\section{The tree set compactness theorem}\label{sec:compactnessthm}

The idea behind describing an algebraic structure as an inverse limit of finite structures lies, quite generally, in the promise to be able to lift properties known for those finite structures to the given profinite one. In the context of tree sets, one particularly interesting such property is that that the tree set is one `over' some collection~$\F$ of sets of separations: tree sets over $\F$ are dual to $\F$-tangles, and the latter provide a very general concept that can be used to capture highly cohesive substructures of a given structure such as a graph or a matroid. In this section we prove such a result: we show that a profinite nested separation system $ \vS=\invlim(\,\vSp\mid p\in P\,) $ is `over'~$\F$ if its projections $\vSp$ are, essentially, `over' the corresponding projections of~$\F$.

Let us define the concepts involved here. We start by lifting the inverse system $(\,\vSp\mid p\in P\,)$ defining~$\vS$ to an inverse system of the power sets $2^\vSp$, whose inverse limit describes the power set $2^\vS\!$ of~$\vS$:

\begin{LEM}\label{powerset}
Let $\S = (\,\vSp\mid p\in P\,)$ be an inverse system with maps~$f_{qp}$.
   \begin{enumerate}[\rm(i)]\itemsep=0pt\vskip-\smallskipamount\vskip0pt
   \item The family $2^\S\! := (\,2^{\vSp}\!\!\mid p\in P\,)$%
   \COMMENT{}
   is an inverse system with respect to $f_{qp}\colon 2^{\vSq}\to 2^{\vSp}$, where $f_{qp}$ maps $\sigma\sub \vSq$ to $\{\,f_{qp} (\vs)\mid \vs\in \sigma\,\}\sub \vSp$.%
   \COMMENT{}
   \item Every limit $(\,\sigma_p\mid p\in P\,)\in\varprojlim 2^\S$ is itself an inverse system with maps $f_{qp}\restricts \sigma_q$: a surjective restriction of~$\S$.\qed
   \end{enumerate}
\end{LEM}

If we write $\vS:=\varprojlim\S$ as usual, there is a natural bijection $\varphi\colon 2^\vS\!\to\varprojlim 2^\S$ given by $\sigma\mapsto (\,\sigma\proj p\mid p\in P\,)$. Its inverse $\varphi^{-1}$ coincides with the $\varprojlim$ operator on the elements of $\varprojlim 2^\S$ according to (ii) above, as $\sigma = \invlim (\,\sigma\proj p\mid p\in P\,)$. We shall use this bijection to topologize~$2^\vS$, giving it the coarsest topology that makes $\varphi$ continuous (where $\varprojlim 2^\S$ carries the product topology with the finite sets~$2^\vSp$ discrete, as usual).

Let us return to our context in which $\vS=\invlim(\,\vSp\mid p\in P\,)$ is a profinite separation system.%
   \COMMENT{}
   For $\F\sub 2^{\vSinf}\!$ we write
 $$\F\restricts p := \{\,\sigma\restricts p\mid \sigma\in\F\,\}$$
 and say that a nested subset $\tau\sub\vS$ is {\em over\/}~$\F$ if all its splitting stars lie in~$\F$. Note that $\F$ is closed in~$2^\vS\!$ if and only if it contains every set $\sigma\sub\vS$ such that for every $p\in P$ there exists some $\tilde\sigma\in\F$ with $\tilde\sigma\restricts p = \sigma\restricts p$.

Naively, our aim would be to prove that $\vS$ has a nested subset over~$\F\in 2^\vS\!$ as soon as $\F$ is closed%
   \footnote{See the proof of Theorem~\ref{treesetcompactness} below for why some such assumption will be necessary.}
   in~$2^\vS\!$ and every $\vSp$ has a nested subset $\tau_p$ over~$\F\proj p$. The standard way to prove this would be to show that all these~$\tau_p$ form an inverse system as a restriction of~$2^\S\!$, where $\S = (\,\vSp\mid p\in P\,)$ as earlier, and then to use the compactness theorem, Lemma~\ref{CompactnessThm}, to find a limit point $(\,\tau_p\mid p\in P\,)$ of this system, one compatible family of nested separation systems $\tau_p$ over~$\F\proj p$. We would then hope to show that the limit $\tau = \invlim (\,\tau_p\mid p\in P\,)$ of {\em this\/} family, a restriction of~$\S$, is a nested subset of~$\vS$ over~$\F$.

An immediate problem with this approach is that we cannot expect to find such nested $\tau_p\sub\vSp$ over~$\F\proj p$: recall that the sets in $\F\proj p$ will usually contain trivial or co-trivial separations, but splitting stars of nested separation systems do not. We shall therefore have to weaken the requirement of being over~$\F\proj p$ to being `essentially' over~$\F\proj p$, that is, up to the deletion of trivial separations.%
   \footnote{Co-trivial separations will not arise in the relevant projections; cf.\ Lemmas \ref{LemmaB} and~\ref{LemmaX}.}

The next problem, then, is in which separation system to measure this triviality. If we measure it in the relevant~$\vSp$, it can happen that the nested separation systems $\tau_p\sub\vSp$ that are essentially over~$\F\proj p$ still do not form an inverse system. (This is because projections of separations $\vrq\in\vSq$ that are trivial in~$\vSq$ can project, for $q>p$, to separations $\vsp\in\vSp$ that are {\em not\/} trivial in~$\vSp$~-- a~subtle but critical issue.%
   \footnote{Indeed, this caused us some headache. The problem is that the projections of $\vrq$ and the witness $\vsq$ of its triviality in~$\vSq$ can project to the same element of~$\vSp$, which then is no longer trivial because it has `lost its witness'. Compare Lemma~\ref{LemmaD} and the discussion preceding it.}%
   )

The solution will be to measure the triviality of separations in the stars $\sigma\in\F\proj p$ not in~$\vSp$ but in these~$\sigma$ themselves. Lemmas \ref{LemmaB} and~\ref{LemmaX} will then help us show that the nested subsystems $\tau_p$ of~$\vSp$ that are essentially over~$\F\proj p$ in this sense, do form an inverse system, and we can proceed as outlined above.

To help with the applicability of our compactness theorem later, we shall also make some provisions for such quirks as finitely trivial elements of~$\vS$, or elements $\vr = (\,\vrp\mid p\in P\,)$ such that for every $p$ there is an inconsistent pair $\sv,\vsdash\in\vS$ with $\sv\proj p = \rvp = \vsdash\proj p$. Let us call such $\rv\in\vS$ {\em finitely inconsistent\/}.

Note that if $\vr = (\,\vrp\mid p\in P\,)$ is finitely trivial in~$\vS$, then all the~$\vrp$ are small.%
   \COMMENT{}
   And the same is true if $\rv$ is finitely inconsistent, since $\vrp = \vs\restricts p \le \vsdash\restricts p = \rvp$ by \eqref{*} and the inconsistency of $\{\sv,\vsdash\}$. By Lemma~\ref{small}, this implies that $\vr$ too is small. Thus, finitely trivial separations are small, and finitely inconsistent separations are co-small.

Recall that $\F$ is closed in~$2^\vS\!$ if and only if it contains every set $\sigma\sub\vS$ such that for every $p\in P$ there exists some $\tilde\sigma\in\F$ with $\tilde\sigma\restricts p = \sigma\restricts p$. Let us call $\F$ {\em essentially closed in~$2^\vS\!$}%
   \COMMENT{}
   if it satisfies the following three conditions. The first is that $\F$ contains every set~$\sigma\sub\vS$ such that for every $p\in P$ there exists $\tilde\sigma\in\F$ such that either $(\tilde\sigma\proj p)^- = (\sigma\proj p)^-$ or $(\tilde\sigma\proj p)^- = \{\vs\!\proj p, \sv\!\proj p\}$ and $\sigma = \{\vs\}$. The second condition is that $\F$ contains every singleton $\{\rv\}$ such that $\rv$ is finitely inconsistent in~$\vS$.%
   \COMMENT{}
   The third condition is that $\F$ contains every singleton $\{\rv\}$ such that $\vr$ is finitely trivial in~$\vS$.

Note that, formally, these conditions depend not only on~$\F$ and~$\vS$ but also on the inverse system $\S=(\,\vSp\mid p\in P\,)$ of which $\vS$ is the inverse limit. While bearing this in mind, it would be cumbersome to keep making this explicit; after all, the topology on~$2^\vS\!$, and hence the notion of $\F$ being (properly) closed in~$2^\vS\!$, also depend on $\S$ itself.%
   \COMMENT{}

   However, we shall need later that if $\F$ is essentially closed with respect to~$\S$ then also with respect to any cofinal subsystem $\S_0 = (\,\vSp\mid p_0\le p\in P\,)$ of~$\S$, where $p_0$ is some fixed element of~$P$,%
   \COMMENT{}
   whose inverse limit will also be~$\vS$.%
   \COMMENT{}
   While it is immediate for the second and third condition in the definition of `essentially closed' that their validity for~$\S$ implies their validity for~$\S_0$,%
   \footnote{If a separation in~$\vS$ is finitely inconsistent with respect to~$\S_0$ then clearly also with respect to~$\S$, and similarly for finitely trivial.}%
   \COMMENT{}
    we need a lemma to ensure this also for the first condition:

\begin{LEM}\label{condition1down}
Let $p < q$, and consider stars $\sigma,\tilde\sigma$ in~$\vS$ such that ${(\sigma\proj q)^- = (\tilde\sigma\proj q)^-}$. Then also $(\sigma\proj p)^- = (\tilde\sigma\proj p)^-$.
\end{LEM}

\begin{proof}
The assertion follows from Lemma~\ref{iteratedminus} applied to both $\sigma$ and~$\tilde\sigma$.%
   \COMMENT{}
\end{proof}

\goodbreak

\begin{LEM}\label{iteratedminus}
For every star $\sigma$ in~$\vS$ and $p<q$ we have $(\sigma\proj p)^- = ( (\sigma\proj q)^-\!\proj p )^-$.
\end{LEM}

\begin{proof}
   By Lemma~\ref{LemmaD} the nontrivial elements of~$\sigma\proj p$, those in~$(\sigma\proj p)^-$, are projections of elements of nontrivial elements of~$\sigma\proj q$, which are those in~$(\sigma\proj q)^-$, and thus lie in~$(\sigma\proj q)^-\!\proj p$. As they are nontrivial even in $\sigma\proj p = (\sigma\proj q)\proj p\supe {(\sigma\proj q)^-\proj p}$,%
   \COMMENT{}
    this shows~`$\sub$'.

For a proof of~`$\supe$', let us start from the inclusion ${(\sigma\proj q)^-\!\proj p}\sub \sigma\proj p$ noted above. Abbreviating $\sigma\proj q =: \rho$, and noting that $\sigma\proj p = (\sigma\proj q)\proj p$, we can rewrite this as $\rho^-\!\proj p\sub \rho\proj p$. We need to show that this implies
 \begin{equation}\label{inclusion}
(\rho^-\!\proj p)^- \sub (\rho\proj p)^-,
 \end{equation}
 i.e., that deleting the respective trivial elements in the stars $\rho^-\!\proj p$ and~$\rho\proj p$ maintains the inclusion. Thus, we have to show that elements of $\rho^-\!\proj p$ that are trivial in the larger star $\rho\proj p$ are trivial also in $\rho^-\!\proj p$ itself.%
   \COMMENT{}

In order to deduce this from Lemma~\ref{LemmaB}, with the $\vS$ there taken to be the separation system $\vR$ formed by $\rho\proj p$ and the inverses of its elements, let us show that $\rho^-\proj p$ contains a splitting star of~$\vR$. Since $\rho\proj p$ is a star it is consistent~\cite{AbstractSepSys},%
   \COMMENT{}
   and the set of its maximal elements splits~$\vR$.%
   \COMMENT{}
   By Lemma~\ref{Remark8} these separations in $\rho\proj p$ are nontrivial in~$\vR$. They therefore have a nontrivial preimage in~$\rho$ by Lemma~\ref{LemmaD}, i.e., lie in~$\rho^-\proj p$ as desired.

If $\rho^-\proj p$ is antisymmetric, then $(\rho^-\proj p)^-$ is equal to this splitting star of~$\vR$ by Lemma~\ref{LemmaB}(i). In particular, $(\rho^-\proj p)^-$ contains no separations that are trivial in~$\vR$,%
   \COMMENT{}
   proving~\eqref{inclusion}. If $\rho^-\proj p$ is not antisymmetric, it has the form $\rho^-\proj p = \{\vs,\sv\}$ by Lemma~\ref{LemmaB}(ii), with $s$ nontrivial in~$\vR$, again proving~\eqref{inclusion}.%
   \COMMENT{}
   \end{proof}

We need one more lemma to deal with the second condition in the definition of `essentially closed':

\begin{LEM}\label{cofinalLp}
Let $\sigma$ be a splitting star of some nested $\tau\sub\vS = \varprojlim\S$, where $\S = (\,\vSp\mid p\in P\,)$. Assume that, for all $p\in P'$ with $P'$ cofinal in~$P$, there exists $\vrp\in\vSp$ such that $(\sigma\proj p)^- = \{\rvp\}$ and $\sv\proj p = \rvp = \vsdash\proj p$ for some inconsistent $\sv,\vsdash\in\vS$.%
   \COMMENT{}
   Then $\sigma = \{\rv\}$ for some $\rv\in\vS$ that is finitely inconsistent in~$\S$.%
   \COMMENT{}
\end{LEM}

\begin{proof}
The sets $\rho_p\sub\sigma\proj p$%
   \COMMENT{}
   of separations $\rvp\in\vSp$ for which there are inconsistent ${\sv,\vsdash\in\vS}$ such that $\sv\proj p = \rvp = \vsdash\proj p$ form an inverse system $(\,\rho_p\mid p\in P\,)$ as a restriction of~$\S$. The $\rho_p$ are nonempty by assumption for $p\in P'$, and are hence nonempty for all~$p$. Pick a limit $\rv = (\,\rvp\mid \in P\,)\in\vS$ with $\rvp\in\rho_p$ for all~$p$.

As noted earlier, the $\rvp\in\rho_p\sub\sigma\proj p$ are co-small, since $\vrp = \vs\restricts p \le \vsdash\restricts p = \rvp$ by \eqref{*} and the inconsistency of $\{\sv,\vsdash\}$, the pair associated with them as in the definition of~$\rho_p$.%
   \COMMENT{}
   By Lemma~\ref{small}, then, $\rv$~is co-small too, and by Lemma~\ref{closed}(i) it lies in the closure of~$\sigma$ in~$\vS$. Let us show that, in fact, $\rv\in\sigma$.

This is clear if $|\sigma|=1$: then $|\rho_p|\le |\sigma\proj p| = 1$ too and hence $\rho_p = \{\rvp\}$ for all~$p$, giving $\rv = (\,\rvp\mid p\in P\,)\in\sigma$.%
   \COMMENT{}
   So assume that $|\sigma|\ge 2$. By Lemma~\ref{lem:boundedclosed}, then, $\rv$~lies in the consistent orientation $O$ of~$\vS$ whose set of maximal elements is~$\sigma$.%
   \COMMENT{}
   If $\rv\notin\sigma$, this means that there exists $\sv\in\sigma$ such that $\rv < \sv$, because $\sigma$ splits~$\vS$. Then $\sv > \rv$ is co-small too. Since $\sigma$ is a star, any other separation in~$\sigma$ would be trivial with witness~$s$. As $\sigma$ splits~$\vS$ and hence has no trivial elements, this means that $\sigma = \{\sv\}$ and hence $|\sigma|=1$, contrary to our assumption.

Thus, $\rv\in\sigma$. Since $\rv$ is co-small, this implies $|\sigma|=1$ and hence $\sigma = \{\rv\}$ as above. As $\rv\in\invlim (\,\rho_p\mid p\in P\,)$ is finitely inconsistent by definition of the~$\rho_p$, this completes our proof.
   \end{proof}

Assume that $\F\sub 2^{\vSinf}\!$ is a set of stars in~$\vS$.%
   \COMMENT{}
   For every $p\in P$, let $\L_p$%
   \COMMENT{}
   be the set of all stars%
   \COMMENT{}
   $\sigma\sub\vSp$ with $\sigma^- = \{\rvp\}$ for some~$\rvp$ such that $\vS$ has an inconsistent subset $\{\sv,\vsdash\}$ with $\!\sv\!\restricts p = \rvp = \vsdash\!\restricts p$. Then let
  $$\F_p := \F\restricts p\cup\L_p\,.$$
Let us say that a nested subsystem $\tau_p\sub\vSp$%
   \COMMENT{}
   is {\em essentially over~$\F_p$} if for every splitting star $\sigma_p$ of~$\tau_p$ there exists some $\sigma\in\F_p$ such that either $\sigma_p = \sigma^-$ or $\sigma_p = \{\vs\}$ and $\sigma^- = \{\vs,\sv\}$ for some $\vs\in\vS$.%
   \COMMENT{}

The nested subsets of~$\vSp$ that are essentially over~$\F_p$ form an inverse system:%
   \COMMENT{}

\begin{LEM}\label{Lqp}
Let $p < q$, and assume that $\vSp$ (and hence~$\vSq$) has no degenerate elements.%
   \COMMENT{}
   If $\tau_q$ is a nested subset of~$\vSq$ essentially over~$\F_q$,%
   \COMMENT{}
   then $\tau_p := \tau_q\proj p$ is a nested subset of~$\vSp$ essentially over~$\F_p$.
\end{LEM}

\begin{proof}
By~\eqref{*}, $\tau_p$~is nested. To prove that it is essentially over~$\F_p$, consider a splitting star $\sigma_p$ of~$\tau_p$. Applying Lemma~\ref{LemmaX} to $\tau_q\to\tau_p$ yields that $\tau_q$~is split by a star~$\sigma_q$ such that $(\sigma_q\proj p)^- = \sigma_p$. As $\tau_q$ is essentially over~$\F_q$, there exists a star $\sigma\in\F_q$ that contains~$\sigma_q$.%
   \COMMENT{}
   Let $\rho := \sigma\proj p$. By~\eqref{*}, $\rho$~is a star in~$\tau_p$. It contains the splitting star~$\sigma_p$, because $\sigma$ contains~$\sigma_q$ and $\sigma_q\proj p$ contains~$\sigma_p$. Hence $\rho^- = \sigma_p$ by Lemma~\ref{LemmaB}(i) if $\rho$ is antisymmetric. If not, then $\rho^- = \{\vs,\sv\}$ with $\sigma_p = \{\vs\}$ by Lemma~\ref{LemmaB}(ii). In either case it only remains to show that $\rho\in\F_p$.%
   \COMMENT{}

If $\sigma\in\F\proj q$, then $\rho\in (\F\proj q)\proj p = \F\proj p\sub \F_p$ as desired. Otherwise $\sigma\in\L_q$, with $\sigma^- = \{\rvq\}$ say, and there exist inconsistent $\sv,\vsdash\in\vS$ such that $\!\sv\!\restricts q = \rvq = \vsdash\!\restricts q$. Let $\rvp := \rvq\proj p$. Then also ${\!\sv\!\restricts p = \rvp = \vsdash\!\restricts p}$\rlap.

By \eqref{*} and the inconsistency of $\{\sv,\vsdash\}$ we have $\vrp = \vs\restricts p \le \vsdash\restricts p = \rvp$, so $\rvp$~is co-small. In particular, $\rvp$~cannot be trivial in~$\rho$, as that would also make it small and hence degenerate.%
   \COMMENT{}
   But $\rvp$ does lie in~$\rho$ by definition,%
   \COMMENT{}
    so~$\rvp\in\rho^-\!=\sigma_p$.

Since $\sigma_p$ is a star, the fact that $\rvp\in\sigma_p$ is co-small implies that any other element of~$\sigma_p$ would be trivial in it. This would contradict Lemma~\ref{Remark8}. Hence $\sigma_p = \{\rvp\}$, and thus again $\rho\in\L_p\sub\F_p$ as desired.
   \end{proof}

Here, then, is our compactness theorem for tree sets, and more generally for nested subsets, in profinite separation systems. Recall that a nested subset $\tau\sub\vS$ is {\em over\/}~$\F$ if all its splitting stars lie in~$\F$.

\begin{THM}\label{treesetcompactness}
Let $\vS=\invlim(\,\vSp\mid p\in P\,)$ be a profinite separation system without degenerate elements, and let $\F\sub 2^{\vSinf}\!$ be essentially closed in~$2^\vS\!$.%
   \COMMENT{}
   If every~$\vSp$ has a nested subset~$\tau_p$ that is essentially over~$\F_p$, then $\vS$ has a closed nested subset~$\tau$ over~$\F\!$.\looseness=-1 %
   \COMMENT{}
\end{THM}

\begin{proof}
Let us show first that, by replacing $P$ with a cofinal subset of the type $\{\,p\in P\mid p\ge p_0\}$ for some fixed $p_0\in P$, and the inverse system $(\,\vSp\mid p\in P\,)$ with the corresponding subsystem (whose inverse limit is still~$\vS$), we may assume that the~$\vSp$ have no degenerate elements. Indeed if every $\vSp$ has a degenerate element then the (non-empty) sets of these, one for each~$p$, form a restriction of our inverse system $(\,\vSp\mid p\in P\,)$ whose limits%
   \COMMENT{}
   are degenerate elements of~$\vS$. This contradicts our assumptions about~$\vS$. Hence there exists $p_0\in P$ such that $\vSpnought$ has no degenerate element. Then the same holds for every $\vSp$ with $p\ge p_0$, and we replace $P$ with the (cofinal) subset of all these~$p$, renaming the subset as~$P\!$.\looseness=-1

Note that replacing $P$ with such a cofinal subset entails no loss of generality. Indeed the premise of our theorem, that $\F$ is essentially closed in~$2^\vS\!$, continues to hold when being essentially closed is redefined with reference to such a subsystem: this is immediate for the second and the third condition in the definition of `essentially closed',%
   \COMMENT{}
   and it follows from Lemma~\ref{condition1down} for the first part of the first condition. For the second part of the first condition, note that if $\sigma = \{\vs\}$ then $(\tilde\sigma\proj p)^- = \{\vs\!\proj p, \sv\!\proj p\}$ follows for $p<q$ from the corresponding assertion for~$q$, with the same~$\tilde\sigma$, since
 \begin{equation}\label{1bdown}
 (\tilde\sigma\proj p)^-\! = ((\tilde\sigma\proj q)\proj p)^-\! = ((\tilde\sigma\proj q)^-\!\proj p)^-\! = (\{\vs\!\proj q, \sv\!\proj q\})\proj p)^-\! = \{\vs\!\proj p, \sv\!\proj p\}^-
 \end{equation}
 by Lemma~\ref{iteratedminus},%
   \COMMENT{}
   and $\{\vs\!\proj p, \sv\!\proj p\}^- = \{\vs\!\proj p, \sv\!\proj p\}$ because $\vs\!\proj p$ and $\sv\!\proj p$ are inverse to each other by~\eqref{*}.%
   \COMMENT{}

 The conclusion for the subsystem, that $\vS$ has a closed nested subset over~$\F$, coincides with that for the original system, since both $\vS$ and its topology%
   \COMMENT{}
   are the same for the original inverse system and its cofinal subsystem.

Let us now prove the assertion of the theorem. If $\tau_p = \es$ for some $p\in P$ then $\es$ is the only splitting subset of~$\tau_p$. As $\tau_p$ is essentially over~$\F_p$, this means that there is a star $\sigma\in\F_p$ such that $\sigma^-=\es$.%
   \COMMENT{}
   But then also $\sigma=\es\,$: since $\sigma$ is finite, it has a nontrivial maximal element if it is nonempty, so that also $\sigma^-\ne\es$.%
   \COMMENT{}
   Thus, $\es\in\F_p$, and hence $\es\in\F\restricts p$.%
   \COMMENT{}
   But this is possible only if $\es$ is also an element of~$\F$.%
    \COMMENT{}
   But then the empty subset of~$\vS$ is over~$\F$,%
   \COMMENT{}
   and we are done. From now on we assume that the sets $\tau_p$ are nonempty.%
   \COMMENT{}

For each $p\in P$, let $\T_p\ne\es$ be the set of all nonempty nested sets $\tau_p\sub\vSp$ that are essentially over~$\F_p$. By Lemma~\ref{Lqp}, $(\,\T_p\mid p\in P\,)$ is an inverse system with respect to the maps $\tau_q\mapsto \tau_q\restricts p$ for~$q > p$. Being subsets of~$\vSp$, the $\tau_p$ contain no degenerate separations.%
   \COMMENT{}

By Lemma~\ref{CompactnessThm}, there exists $(\,\tau_p\mid p\in P\,)\in\, \varprojlim\,(\,\T_p\mid p\in P\,)$.%
   \COMMENT{}
   This family ${(\,\tau_p\mid p\in P\,)}$ is itself a surjective inverse system, a restriction of $(\,\vSp\mid p\in P\,)$. The set
 $$\tau :=\, \varprojlim\,(\,\tau_p\mid p\in P\,)$$
 of its limits is closed in~$\vS$ by Lemma~\ref{closed}, and nested by Lemma~\ref{lemma3}. Note that%
   \COMMENT{}
   $\tau\restricts p = \tau_p = \tau_q\restricts p$ for all $p,q\in P$ with $p\le q$.

For a proof that $\tau$ is over~$\F$ we have to show that every $\sigma\sub\tau$ that splits~$\tau$ lies in~$\F$.%
   \COMMENT{}
   As $\F$ is, by assumption, essentially closed this is the case if $\sigma=\{\rv\}$ with $\vr$ finitely trivial in~$\vS$. We may therefore assume that $\sigma$ is not of this form. By Lemma~\ref{lemma5}, therefore, and our assumption that $\vS$ has no degenerate elements,%
   \COMMENT{}
   there exists $p_1\in P$ such that every $\sigma_p^\circ := (\sigma\restricts p)\cap\tau_p^\circ$
 with $p \ge p_1$ splits~$\tau_p$.%
   \COMMENT{}

The star $(\sigma\proj p)^-\!\supe\sigma_p^\circ$%
   \COMMENT{}
   must be antisymmetric: otherwise Lemma~\ref{LemmaB}(ii) would imply that $\sigma_p^\circ = \{\vs\}\subsetneq \{\vs,\sv\} = (\sigma\proj p)^-$ for some $\vs\in\tau_p^\circ$, in which case $\sv$ would also lie in~$\tau_p^\circ$ (because $\vs$ does) and hence  in $(\sigma\proj p)\cap\tau_p^\circ = \sigma_p^\circ$, a~contradiction. By Lemma~\ref{LemmaB}(i),%
   \COMMENT{}
   therefore, we have
 \begin{equation}\label{blabla}
   \sigma_p^\circ%
   \COMMENT{}
    = (\sigma\proj p)^-.
\end{equation}

As $\sigma_p^\circ$ splits~$\tau_p$, which is essentially over~$\F_p$, there exists $\tilde\sigma_p\in\F_p$ such that either
 \begin{equation}\label{bla}
   \tilde\sigma_p^- = \sigma_p^\circ%
   \COMMENT{}
 \end{equation}
or
 \begin{equation}\label{ddag}
   \tilde\sigma_p^- = \{\vsp,\svp\}\text{ with }\sigma_p^\circ = \{\vsp\}
 \end{equation}
 for some $\vsp\in\tau_p^\circ$.

Suppose first that these $\tilde\sigma_p$ can, for some set of $p\ge p_1$ cofinal in~$P$, be found in~$\L_p$ rather than in~$\F\restricts p$. By definition of~$\L_p$, this means that $|\tilde\sigma_p^-| = 1$ for all those~$p$, so $\tilde\sigma_p$ satisfies \eqref{bla} rather than~\eqref{ddag}. Thus,
 \begin{equation*}
 (\sigma\proj p)^- \specrel={\eqref{blabla}} \sigma_p^\circ \specrel={\eqref{bla}} \tilde\sigma_p^- = \{\rvp\}
 \end{equation*}
 with inconsistent $\sv,\vsdash\in\vS$, depending on~$p$, such that $\!\sv\!\restricts p = \rvp = \vsdash\!\restricts p$. By Lemma~\ref{cofinalLp} this yields that $\sigma = \{\rv\}$ with $\rv$ finitely inconsistent, which in turn implies $\sigma\in\F$ since $\F$ is essentially closed.

We may now assume that, for some set $P'$ of $p\ge p_1$ cofinal in~$P$, we have $\tilde\sigma_p\in\F\proj p$, and thus $\tilde\sigma_p = \tilde\sigma\proj p$ for some~$\tilde\sigma\in\F$ depending on~$p$. We may assume that either \eqref{bla} or~\eqref{ddag} holds for {\em all\/}%
   \COMMENT{}
   $p\in P'$ since one of them must occur cofinally often. Suppose first that \eqref{bla}~holds for all~$p$. By \eqref{bla}, \eqref{blabla} and~$\tilde\sigma_p = \tilde\sigma\proj p$ we then have
  $$(\tilde\sigma\proj p)^- = \,(\sigma\proj p)^-$$
for all $p\in P'$. By Lemma~\ref{condition1down} this identity holds for all~$p\in P$, not just those in~$P'$. But this implies that $\sigma\in\F$ since $\F$ is essentially closed in~$2^\vS\!$.%
   \COMMENT{}

Suppose finally that \eqref{ddag} holds for all $p\in P'$. Then $\{\vsp\} = \sigma_p^\circ = (\sigma\proj p)\cap\tau_p^\circ$%
   \COMMENT{}
   for all $p\in P'$. So each $\sigma\proj p$ has only one element that is not trivial in~$\tau_p$. Lemma~\ref{LemmaD} therefore implies that these $\vsp$ are compatible, in that $f_{qp}(\vsq) = \vsp$ for all $p < q$ in~$P'$. As $P'$ is confinal in~$P$, these $\vsp$ (and their projections to $p\in P\sm P'$)%
   \COMMENT{}
   form an element $\vs := (\,\vsp\mid p\in P\,)$ of~$\vS$. Note that $\vs\in\tau$ by Lemma~\ref{closed}(i), since $\tau$ is closed in~$\vS$.

Let us show that $\sigma = \{\vs\}$. Suppose not, and consider any $\vr\in\sigma\sm\{\vs\}$. Then $\vrp := \vr\proj p\ne\vs\proj p = \vsp$ for cofinally many~$p$.%
   \COMMENT{}
   But $\vrp\ne\vsp$ means that $\vrp$ is trivial in~$\sigma\proj p$ (and hence in~$\tau_p$), as otherwise  it would be in
 $$(\sigma\proj p)^-\! \specrel={\eqref{blabla}} \sigma_p^\circ \specrel={\eqref{ddag}} \{\vsp\}\,,$$
 a~contradiction.%
   \COMMENT{}

As every trivial separation in a finite separation system has a nontrivial witness, and $\vsp$ is the only nontrivial element of~$\sigma\proj p$, every $\vrp$ is trivial in $\sigma\proj p$ with witness~$s_p$. Thus, $\vrp\le\vsp$ as well as $\vrp\le\svp$ for all our cofinally many~$p$,%
   \COMMENT{}
   so $\vr\le\vs$ and $\vr\le\sv$ by definition of~$\le$ in~$\vS$ and~\eqref{*}.

If $r\ne s$, this shows that $\vr$ is trivial in~$\tau$. (Here we use that $\vs\in\tau$; see earlier.) This contradicts our assumption of $\vr\in\sigma$ by Lemma~\ref{Remark8}, since $\sigma$ splits~$\tau$. Hence $r=s$, and thus $\vr = \sv$ by the choice of~$\vr$. As $\vr$ was an arbitrary element of $\sigma\sm \{\vs\}$, this means that $\sigma\sm\{\vs\} = \{\vr\} = \{\sv\}$. Being a splitting star, $\sigma$~is antisymmetric,%
   \COMMENT{}
   so our assumption of $\vr\in\sigma$ implies by $\vr=\sv$ that $\vs\notin\sigma$. Hence, in fact, $\sigma=\sigma\sm\{\vs\} = \{\sv\}$.

As $\sigma = \{\sv\}$ and $\sv\!\proj p = \svp$ by definition of~$\vs$, we thus have
 $$\{\svp\} = \sigma\proj p = (\sigma\proj p)^-\! \specrel={\eqref{blabla}} \sigma_p^\circ \specrel={\eqref{ddag}} \{\vsp\}\,.$$%
   \COMMENT{}
   So $\vsp$ is degenerate, which contradicts our assumption about~$\vSp$ secured early in the proof. This completes our proof that $\sigma=\{\vs\}$.

Recall that, by assumption and~\eqref{ddag},%
   \COMMENT{}
   there exists for every $p\in P'$ some $\tilde\sigma\in\F$ such that $(\tilde\sigma\proj p)^- = \{\vsp,\svp\} = \{\vs\!\proj p, \sv\!\proj p\}$. As $P'$ is cofinal in~$P$, the same then holds for all $p\in P$; see~\eqref{1bdown}. Since $\F$ is essentially closed, this implies $\sigma\in\F$ as desired.
   \end{proof}

Finally, let us address the question of when the nested separation system $\tau\sub\vS$ found by Theorem~\ref{treesetcompactness} can be chosen to be a tree set. The essential core~$\tau^\circ$ of~$\tau$, obtained from~$\tau$ by deleting all its trivial and co-trivial separations,%
\COMMENT{}
is indeed a tree set%
\COMMENT{}
that is still over~$\F$: note that any splitting star~$\sigma$ of~$\tau^\circ$ also splits~$\tau$, because adding the trivial orientations of any separation in $\tau\sm\tau^\circ$ to a consistent orientation of~$\tau^\circ$ yields a consistent orientation of~$\tau$ with the same maximal elements.%
\COMMENT{}
The question, however, is whether $\tau^\circ$ is closed in~$\tau$ (and hence in~$\vS$, as required by the theorem).

In general,%
   \COMMENT{}
   it is possible that the essential core of a nested separation system fails to be closed in it: Example~\ref{ex:ray} shows one of the ways in which this can happen, and Example~\ref{ex:splittingnotclosed2} below describes another instance.

The proof of Theorem~\ref{treesetcompactness} cannot easily be altered to prevent this from happening to the set~$\tau$ it finds: given $ p\in P $ and a nested subset $ \tau_p $ of $ \vSp $ that is essentially over $ \F_p $, the essential core $ \tau_p^\circ $ will be a tree set essentially over $ \F_p $. However we might not be able to find a compatible family $ (\,\tau_p^\circ\mid p\in P\,) $ of such tree sets: for $ q>p $ a nontrivial separation from $ \vSq $ may well project to trivial separations in $ \vSp $, so $ \tau_q\proj p=\tau_p $ does not imply that $ \tau_q^\circ\proj p=\tau_p^\circ $.

\section{Orienting profinite nested separation systems}\label{sec:orientations}

In Section~\ref{sec:limtreesets} we examined how the splitting stars of a nested profinite separation system $ \vS=\invlim(\,\vSp\mid p\in P\,) $ relate to those of its projections. Let us now look at the consistent orientations which these splitting stars induce.

If $\vS$ is finite, all its consistent orientations~$O$ are induced by a splitting star~$\sigma$, in the sense that $O = \dcl(\sigma)\sm\amgis$, where $\sigma$ is the set of maximal elements~of~$O$ (and $\amgis := \{\sv\mid\vs\in\sigma\}$).%
   \COMMENT{}

This is not the case for profinite separation systems in general. For example, the orientations of the edge tree set $\vec E(T)$ of an infinite tree~$T$ towards some fixed end of~$T$ are consistent but do not split. Note that these orientations also fail to be closed in the topology of the inverse limit: by Corollary~\ref{closedsplits}, every closed consistent orientation of~$ \vS $ splits. This raises the question of whether the closed consistent orientations of a nested profinite separation system $ \vS $ are precisely those that split. Such a purely combinatorial description of the closed consistent orientations of~$ \vS $, and in addition one that makes no reference to the inverse system $(\,\vSp\mid p\in P\,) $, would certainly be useful.

Unfortunately, though, even tree sets can have consistent orientations that split but are not closed. Here is an example that is not a tree set, but typical:%
   \COMMENT{}

\begin{EX}\label{ex:splittingnotclosed}
	Let $ P$ be the set of all integers $p\ge 3 $. For each $ p\in P $ let $\vSp$ be the separation system consisting of a star $\sigma_p$ with $p$ elements, plus their inverses, where $\sigma_p$ contains two separations $\vsp$ and $\vrp$ such that $\sigma_p\sm\{\vsp,\vrp\}$ is a proper star, $\vsp\le\svp$, and $\vrp < \vt$ (as well as $\vrp < \tv$) for all $\vt\in\sigma_p\sm\{\vrp\}$. For each $p$ pick an element $\vtp$ of $\sigma_p\sm\{\vsp,\vrp\}$.
	
	Let $ f_{p+1,p}\colon\vSp{}_{+1}\to\vSp $ be the homomorphism which maps~$\vsp{}_{+1} $ and $\vtp{}_{+1} $ to~$ \vsp $, maps $\vrp{}_{+1}$ to $\vrp$, maps $ \sigma_{p+1}\sm\{\vsp{}_{+1},\vrp{}_{+1},\vtp{}_{+1}\}$ bijectively to~$ \sigma_p\sm\{\vsp,\vrp\} $, and is defined on the inverses of these separations so as to commute with the inversion. These maps induce bonding maps $ f_{qp}\colon\vSq\to\vSp $ by concatenation.
	
	Let $\vS:=\invlim (\,\vSp\mid p\in P\,) $. Then $\vS$ consists of an infinite star~$\sigma$%
   \COMMENT{}
   as well as the respective inverses, and we have $\sigma\proj p=\sigma_p$ for all $p\in P$. The separation $\vs=(\vsp\mid p\in P)$ lies in $\sigma$ and is small, as $\vsp\le\svp$ for all $p\in P$ by construction. However $\vs$ is not trivial in $\vS$: we have $\vs\not<\vr$ for $\vr:=(\vrp\mid p\in P)\in\sigma$, since $\vrp<\vsp$ for all~$p$; and for every other $\vt\in\sigma$ with $\vt\ne\vs$ there is a unique~${p\in P}$ such that $\vt\proj p=\vtp$, and for this $p$ we have $\vsp\not<\vtp$. Hence $\vs\not<\vt$ in $\vS$.\COMMENT{}

For future use, note that $\vr$ is trivial in~$\vS$, witnessed by~$s$.
	
	Consider the orientation $O:=\{\sv\}\cup(\sigma\sm\{\vs\})$ of~$\vS$. By the star property of $\sigma$ we have $\vt\le\sv$ for all $\vt\in\sigma$ with $\vt\ne\vs$, so $\sv$ is the greatest element of~$O$. As $\sv$ is not co-trivial in~$\vS$, as seen above, this means that $O$ is consistent and splits at~$\{\sv\}$.
	
	Let us now show that $O$ is not closed in $\vS$ by proving that $\vs$ lies in the closure of $O$ in $\vS$. By Lemma~\ref{closed} we need to show that $\vsp\in O\proj p$ for each $p\in P$. So let $p\in P$ be given and consider the (unique) element $\vt\in\sigma\sm\{\vs\}$ for which $\vt\proj (p+1)=\vtp_{+1}$. Then $\vt\proj p=\vsp$ by the definition of $f_{p+1,p}$ and hence $\vsp\in O\proj p$.
	
	Therefore the orientation $O$ of~$\vS$ splits but is not closed in $\vS$.\qed
\end{EX}

The next two lemmas show that every orientation $ O $ of $ \vS $ that splits but fails to be closed must more or less look like Example~\ref{ex:splittingnotclosed}. First we show that if $ O $ has a greatest element then this must be co-small:

\begin{LEM}\label{lem:greatestclosed}
Let $O$ be a consistent orientation of $\vS$ with a greatest element~$\vm$. Then either $ O $ is closed in $ \vS $, or $ \vm $ is co-small but not co-trivial and $ \mv $ lies in the closure of~$ O $.
\end{LEM}

\begin{proof}
Note first that $\vm$ cannot be co-trivial, because consistent orientations of separation systems cannot have co-trivial elements (Lemma~\ref{lem:extension}\,(i)).
Suppose that $ O $ fails to be closed, that is, there is some nondegenerate $ \vs\in O $ for which $ \sv $ lies in the closure of~$ O $ in $ \vS $. By Lemma~\ref{closed} there exists for every $ p\in P $ some $ \vr=\vr(p)\in O $ such that $ \sv\restricts p=\vr\restricts p $. As $ \vm $ is the greatest element of $ O $ we have $ \vr\le\vm $, and hence $ \sv\restricts p=\vr\restricts p\le\vm\restricts p $ for every~$p$, giving $ \sv\le\vm $ in $ \vS $. If $ s\ne m $ this contradicts the consistency of~$ O $. So $ s=m $, and hence $ \vs=\vm $ since these lie in~$O$. Thus, $\vs = \vm \ge \sv$ is co-small.
\end{proof}

If a consistent orientation of $\vS$ has two or more maximal elements, then these cannot be co-small. Indeed, since $\vS$ is nested, the only way such elements could be incomparable would be that they are inconsistent.%
   \COMMENT{}
   Thus if such an orientation splits, we would expect it to be closed in~$\vS$. This is indeed the case:

\begin{LEM}\label{lem:boundedclosed}
Any consistent orientation~$O$ of~$\vS$ that splits at a star of order at least~2 is closed in~$ \vS $.
\end{LEM}

\begin{proof}
We verify the premise of Lemma~\ref{closed}\,(ii), with $O_p:= O\proj p$ for all~$p$.

Suppose that $O\ne O':=\invlim (\,O_p\mid p\in P\,)$. Then $O'$ contains some ${\sv\notin O}$, so~$\vs\in O$. This $\vs$ lies below some maximal element~$\vm$ of~$O$. Let $ \vn $ be another maximal element of~$ O $.  As $\vm$ and $\vn$ are consistent and incomparable, but have comparable orientations,%
   \COMMENT{}
   they point towards each other. In particular, $\vn\le\mv$. As $\vm$ and $\vn$ both lie in~$ O$ they cannot be inverse to each other, so $\vn < \mv$. Hence there exists $q\in P$ such that $ \vn\proj p < \mv\proj p $ for all $ p\ge q$.%
   \COMMENT{}

Consider any~$ p\ge q $. As $ \sv\in O' $ there exists $ \vr=\vr(p)\in O $ with $ \vr\proj p=\sv\proj p $. This $ \vr $ lies below a maximal element of~$ O $. But we cannot have $ \vr\le\vn $, since in that case $ \vn\proj p\ge\vr\proj p=\sv\proj p\ge\mv\proj p $, contradicting the fact that $p\ge q$.

Therefore $ \vr $ lies below some other maximal element of $ O $ and hence, like that element, points towards $ \vn $: we have $ \vr\le\nv $. Thus, $ {\nv\proj p}\ge {\vr\proj p} = {\sv\proj p} \ge {\mv\proj p} $. As this inequality holds for each $ p\ge q $ we have $ \vn\le\vm $ in $ \vS $, contrary to the maximality of $ \vn\ne\vm $ in~$ O $.
\end{proof}

Lemmas~\ref{lem:greatestclosed} and~\ref{lem:boundedclosed} together show that a consistent orientation of $ \vS $ that splits but is not closed in $\vS$ has a very particular form: it must have a co-small greatest element that is not co-trivial. Rather than describing such orientations directly, let us see if we can characterize those $ \vS $ that admit such an orientation.

Let us call a nested separation system $\vS$ {\em normal\/} if the consistent orientations of~$\vS$ that split are precisely those that are closed in~$\vS$.

\begin{THM}\label{thm:regular}
    If all small separations in $\vS$ are trivial, then $\vS$ is normal. In particular, regular profinite tree sets are normal.
\end{THM}

\begin{proof}
	Let us assume that all small separations in $\vS$ are trivial and show that $\vS$ is normal. By Corollary~\ref{closedsplits}, all closed consistent orientations of~$\vS$ split. Conversely, let $O$ be a consistent orientation of $\vS$ that splits, say at~$\sigma$. If $\sigma$ is empty then so is~$\vS$, so $O$ is closed in~$\vS$. If $\abs{\sigma}\ge 2$ then $O$ is closed by Lemma~\ref{lem:boundedclosed}. And if $\abs{\sigma}=1$, then $O$ is closed by Lemma~\ref{lem:greatestclosed}.
\end{proof}

Theorem~\ref{thm:regular} leaves us with the problem to characterize the normal nested separation systems among those that do contain nontrivial small separations. Interestingly, there can be no such characterization, at least not in terms of separation systems alone: normality is not an invariant of separation systems!

Indeed, our next example shows that we can have isomorphic profinite nested separations systems of which one is normal and the other is not:%
   \COMMENT{}

\begin{EX}\label{ex:splittingnotclosed2}
Let $ G=(V,E) $ be a countably infinite star with centre $ z $ and $ x\ne z $ another vertex of $ G $. Let $ \vS_G $ be the separation system consisting of $(\{z\},V)$ and all separations of the form $(\{y,z\},V-y)$ for $y\ne z$, plus inverses. Let $ (A,B)\le(C,D) $ if either $ A\sub C $ and $ B\supseteq D $, or ${(A,B)=(\{x,z\},V-x)}$ and $ (C,D)=(V-x,\{z,x\}) $.

Let $ Q $ be the set of all finite subsets of $ V $ that contain $ x $ and $ z $. For all $ q\in Q $ let $ \vSq $ be the set of all $ (A,B)\proj q:=(A\cap q\,,\,B\cap q) $ with $ (A,B)\in\vS_G $. Given $ (A_q,B_q), (C_q,D_q) \in\vSq$ let $ (A_q,B_q)\le(C_q,D_q) $ if and only if there are $ (A,B)\le(C,D)$ in $\vS_G $ such that $ (A,B)\proj q=(A_q,B_q) $ and $ (C,D)\proj q=(C_q,D_q) $.

Then $ \vS_G=\invlim(\,\vSq\mid q\in Q\,) $. Note that $ \vS_G $ contains only two co-small separations: $ (V-x,\{z,x\}) $ and $(V,\{z\}) $. The first of these is a splitting singleton star, since its down-closure (minus its inverse) orients~$\vS_G$ consistently: the orientation consists of all $(A,B)$ with $x\in B$. This orientation of $\vS_G$ is clearly closed in~$\vS_G$.%
   \COMMENT{}
   In particular, the inverse $(\{x,z\},V-x)$ of its greatest element does not lie in its closure.

The other co-small separation, $ (\{z\},V) $, is co-trivial.%
   \COMMENT{}
   By Lemmas \ref{lem:greatestclosed} and~\ref{lem:boundedclosed}, therefore, all splitting consistent orientations of $ \vS_G $ are closed, so $\vS_G$ is normal.%
   \COMMENT{}

However, $ \vS_G $ is isomorphic to the separation system $ \vS $ of Example~\ref{ex:splittingnotclosed}, which has a consistent orientation that splits but is not closed, and hence is not normal. Indeed, both separation systems consist of a countably infinite star (plus inverses) that contains one small but nontrivial separation, and one trivial separation with all the other separations as witnesses, and has no further relations.\qed
\end{EX}

Example~\ref{ex:splittingnotclosed2} may serve as a reminder that the topology on a profinite separation system~$\vS$ depends, by definition, on the inverse system of which $\vS$ is the inverse limit. It shows that this dependence is not just formal:

\begin{COR}
The topologies of isomorphic profinite nested separation systems can differ.\qed
\end{COR}

In our particular case of Examples~\ref{ex:splittingnotclosed} and~\ref{ex:splittingnotclosed2}, the difference in the topologies hinged on the question of whether the small but nontrivial element of an infinite star lies in the closure of the rest of that star. In Theorem~\ref{thm:regular} we noted that if $\vS$ contains no nontrivial small separations it is normal.

In the rest of this section we shall prove a converse of this: if $\vS$ contains no trivial (small) separation, i.e., is a tree set, it fails to be normal as soon as it contains an infinite star. The reason, interestingly, is that {\em every\/} maximal infinite star contains a small separation that lies in the closure of its other separations~\cite{KneipMaster}. The only way in which this is compatible with normality is that that separation is not only small but in fact trivial.

Recall that $ \vS=\invlim(\,\vSp\mid p\in P\,) $ is a nested profinite separation system.

\begin{LEM}\label{non-c-star}
If $ \vS$ is a tree set containing an infinite star, then $\vS$ is not normal.
\end{LEM}

\begin{proof}
Let $\sigma\sub\vS$ be an infinite star. Like all stars, $\sigma$~is consistent~\cite{AbstractSepSys}. It is also antisymmetric, since otherwise all but two of its elements would be trivial. The inverses of the elements of~$\sigma$, therefore, are pairwise inconsistent.

Our plan is to find a small separation $ \vso\in\vS $ that lies in the closure of $\sigma':=\sigma\sm\{\vso\} $. By Lemma~\ref{lem:extension} we can extend $\{\svo\}$ to a consistent orientation~$ O $ of~$\vS$. Since the inverses of the separations in~$\sigma'$ are pairwise inconsistent, $O$~contains all but at most one element of~$\sigma'$.%
   \footnote{In fact, it contains all of them~-- but this is harder to show~\cite{KneipMaster}.}
   Like~$\sigma'$, therefore, $O$~will have~$ \vso $ in its closure, and hence not be closed. As $ \svo $ is co-small it will be maximal in~$ \vS $, since any larger separation would be co-trivial in~$ \vS $.%
   \COMMENT{}
   Since every two separations in $O$ have comparable orientations, the maximality of $ \svo $ in~$\vS$ implies that $O$ lies in its down-closure and hence splits at~$\{\svo\}$.%
   \COMMENT{}
   Thus, $O$~will be a consistent orientation of~$\vS$ that splits but is not closed in~$\vS$, completing our proof that $\vS$ is not normal.

Let us now show that such an $ \vso $ exists. For every $ p\in P $ let $ \vSdashp\sub\vSp $ be the set of all $ \vs\in\vSp $ for which there are infinitely many $ \vr\in\sigma $ with $ \vr\proj p=\vs $. As $ \sigma $ is infinite but $ \vSp $ is finite, $ \vSdashp $~is non-empty. Note that $ (\,\vSdashp\mid p\in P\,) $ is a restriction%
   \footnote{As an inverse system of sets, not of separation systems: the $ \vSdashp $ need not be closed under taking inverses.}
   of $ (\,\vSp\mid p\in P\,)$. Pick $ \vso\in\invlim(\,\vSdashp\mid p\in P\,) $. Then for each $ p\in P $ there are distinct $ \vr,\vrdash\in\sigma\sm\{\vso\} $ such that $ \vr\proj p=\vso\proj p=\vrdash\proj p $, so $ \vso $ lies in the closure of $ \sigma\sm\{\vso\} $. Furthermore by the star property we have $ \vr\le\rvdash $ and hence $ \vso\proj p\le\svo\proj p $ by~\eqref{*}, so $ \vso $ is small by Lemma~\ref{small}.
\end{proof}

The converse of Lemma~\ref{non-c-star} holds too:

\begin{LEM}\label{star-non-c}
If $\vS$ is a tree set but not normal, it contains an infinite star.
\end{LEM}

\begin{proof}
Assume that $\vS$ is not normal, and let $O$ be a consistent orientation of~$\vS$ that witnesses this: one that splits but is not closed in~$\vS$. By Lemma~\ref{lem:greatestclosed} and~\ref{lem:boundedclosed},%
   \COMMENT{}
   $O$~has a greatest element $ \vm $ with $ \vm\ge\mv $ and $ \mv $ in the closure of $ O $. Suppose that $ \vS $ contains no infinite star. We will find some $ p\in P $ for which $ (\mv\proj p)\notin O\proj p $, contradicting the fact that $ \mv $ lies in the closure of~$O$.

Let $ M $ be the set of minimal elements of $ \vS $. Then $ M $ is a star in~$\vS$,%
   \COMMENT{}
   and therefore finite by assumption. Note that $\mv\in M\sm O$, since any $\sv < \mv$ would have $\vs > \vm$ contradict the definition of~$\vm$, but $M\sm\{\mv\}\sub \dcl(\vm)\sm\{\mv\} = O$.%
   \COMMENT{}
   By Lemma~\ref{lem:closedsplitting} every separation in $ \vS $ lies above an element of $ M $.

No $ \vs\in M $ is trivial in $ \vS $ with witness $ m $. Hence for every $ \vs\in M $ with $ s\ne m $%
   \COMMENT{}
   there is a $ p(\vs)\in P $ such that $ s\proj p\ne m\proj p $ and $ \vs\proj p $ is not trivial with witness $ m\proj p $ for all $ p\ge p(\vs)$.%
   \COMMENT{}
   By~\eqref{*} we have $ \vs\proj p<\vm\proj p $ for all $ p\ge p(\vs) $. Pick $ p\in P $ large enough that $\mv\proj p\ne\vm\proj p$%
   \COMMENT{}
   and $ p\ge p(\vs) $ for all $ \vs\in M $ with $ s\ne m $. We shall show that $ (\mv\proj p)\notin O\proj p $.

Suppose to the contrary that there exists $ \vt\in O $ such that $ \vt\proj p=\mv\proj p $. Pick $\vs\in M$ with $\vs\le\vt$. Then $\vs\le\vt\le \vm$,%
   \COMMENT{}
   where the last inequality is strict since $\vt\proj p = \mv\proj p\ne \vm\proj p$ by the choice of~$p$. In particular, $\vs \ne \vm$. But we also have $\vs\ne\mv$: otherwise $\vs=\mv < \vt < \vm$, where the first inequality is strict since $\vt\in O$ while $\mv\notin O$,%
   \COMMENT{}
   which makes $\mv$ trivial with witness~$t$, a contradiction. Thus, $\vs\notin\{\vm,\mv\}$ and hence $s\ne m$.

By the choice of~$p$, this implies that even $ s\proj p\ne m\proj p $,%
   \COMMENT{}
   yet that $\vs\proj p $ is not trivial with witness $ m\proj p $. But it is, since $ \vs\proj p\le\vm\proj p $ (by $\vs\le\vt\le\vm$) as well as $ \vs\proj p\le\vt\proj p=\mv\proj p $ (by the choice of~$\vt$), a contradiction.
\end{proof}

Lemmas \ref{non-c-star} and~\ref{star-non-c} together amount to a characterization of the profinite tree sets that are normal. It turns out that, for tree sets, normality is more restrictive than it might have appeared at first glance:

\begin{THM}\label{prop:closedsplittingchar}
A profinite tree set is normal if and only if it contains no infinite star.\qed
\end{THM}

\bibliographystyle{plain}
\bibliography{collective.bib}

\end{document}